\documentclass[oneside,openany,a4paper]{amsart}
\usepackage[a4paper, margin=1in]{geometry}
\usepackage{amssymb}
\usepackage[centertags]{amsmath}
\usepackage{amsthm}
\usepackage{amsfonts}
\usepackage{eucal}
\usepackage{mathrsfs}
\usepackage{enumerate}
\usepackage{mathtools}
\usepackage{pst-node}
\usepackage{tikz-cd} 

\usepackage[all,cmtip]{xy}
\usepackage{graphicx, eepic}

\usepackage{lmodern}
\usepackage[utf8]{inputenc}
\usepackage[T1]{fontenc}
\usepackage{pdfrender, xcolor}

\usepackage{bbm}
\usepackage{pdfpages}
\usepackage{titletoc}
\usepackage{booktabs}
\usepackage[shortlabels]{enumitem}
\usepackage{mathtools}
\usepackage{thmtools}

\usepackage{epstopdf}
\usepackage{epsfig}
\usepackage{microtype}

\usepackage{hyperref}
\hypersetup{colorlinks}
\hypersetup{
    bookmarks=true,         % show bookmarks bar?
    unicode=false,          % non-Latin characters in Acrobat’s bookmarks
    pdftoolbar=true,        % show Acrobat’s toolbar?
    pdfmenubar=true,        % show Acrobat’s menu?
    pdffitwindow=false,     % window fit to page when opened
    pdfstartview={FitH},    % fits the width of the page to the window
    pdftitle={My title},    % title
    pdfauthor={Author},     % author
    pdfsubject={Subject},   % subject of the document
    pdfcreator={Creator},   % creator of the document
    pdfproducer={Producer}, % producer of the document
    pdfkeywords={keyword1} {key2} {key3}, % list of keywords
    pdfnewwindow=true,      % links in new PDF window
    colorlinks=true,       % false: boxed links; true: colored links
    linkcolor=red,          % color of internal links (change box color with linkbordercolor)
    citecolor=magenta,        % color of links to bibliography
    filecolor=violet,      % color of file links
    urlcolor=blue      % color of external links
}
\usepackage{multirow, url}
\usepackage{textcomp}
\DeclareMathAlphabet{\cmcal}{OMS}{cmsy}{m}{n}
\newtheoremstyle{thm}% name of the style to be used
  {3pt}% measure of space to leave above the theorem. E.g.: 3pt
  {3pt}% measure of space to leave below the theorem. E.g.: 3pt
  {\em}% name of font to use in the body of the theorem
  {0pt}% measure of space to indent
  {\bfseries}% name of head font
  {}% punctuation between head and body
  {5pt}% space after theorem head
  {}% Manually specify head
\newtheoremstyle{rem}% name of the style to be used
  {3pt}% measure of space to leave above the theorem. E.g.: 3pt
  {3pt}% measure of space to leave below the theorem. E.g.: 3pt
  {}% name of font to use in the body of the theorem
  {0pt}% measure of space to indent
  {\bfseries}% name of head font
  {.}% punctuation between head and body
  {5pt}% space after theorem head
  {}% Manually specify head

\newtheorem{thm}{Theorem}[section]
\newtheorem{cor}[thm]{Corollary}
\newtheorem{lem}[thm]{Lemma}
\newtheorem{prop}[thm]{Proposition}

\newtheorem{conj}[thm]{Conjecture}

\theoremstyle{definition}
\newtheorem{defn}[thm]{Definition}

\theoremstyle{rem}
\newtheorem{rem}[thm]{{Remark}}

\numberwithin{equation}{section} \numberwithin{table}{section}
	
\newtheorem*{thm*}{Theorem}
\newtheorem*{rem*}{Remark}
\newtheorem*{rems*}{Remarks}
\newtheorem*{exam*}{Example}
\newtheorem*{exams*}{Examples}

\makeatletter
\newcommand{\neutralize}[1]{\expandafter\let\csname c@#1\endcsname\count@}
\makeatother

\usepackage{enumerate}

\def\bos#1{{\mathbf{#1}}}

\newcommand{\dlog}{\mathrm{dlog}}
\newcommand{\on}{\operatorname}

  \newcommand{\nc}{\newcommand}
  \newcommand{\be}{\begin{eqnarray*}}
  \newcommand{\ee}{\end{eqnarray*}}
  \newcommand{\bea}{\begin{eqnarray}}
  \newcommand{\eea}{\end{eqnarray}}
  \newcommand{\bs}{\begin{split}}
  \newcommand{\es}{\end{split}}
  \newcommand{\bal}{\begin{align}}
  \newcommand{\eal}{\end{align}}

   \nc{\bei}{\begin{itemize}}
   \nc{\eei}{\end{itemize}}
   \nc{\bee}{\begin{enumerate}}
   \nc{\eee}{\end{enumerate}}
   \nc{\bet}{\begin{thm}}
   \nc{\eet}{\end{thm}}
   \nc{\bed}{\begin{defn}}
   \nc{\eed}{\end{defn}}
   \nc{\bel}{\begin{lem}}
   \nc{\eel}{\end{lem}}
   \nc{\bep}{\begin{prop}}
   \nc{\eep}{\end{prop}}
   \nc{\bec}{\begin{corollary}}
   \nc{\eec}{\end{corollary}}
   \nc{\ber}{\begin{rem}}
   \nc{\eer}{\end{rem}}
   \nc{\beex}{\begin{example}}
   \nc{\eeex}{\end{example}}

   \nc{\bpm}{\begin{pmatrix}}
   \nc{\epm}{\end{pmatrix}}
   \nc{\bspm}{\left(\begin{smallmatrix}}
   \nc{\espm}{\end{smallmatrix}\right)}

\newcommand{\cC}{\mathcal{C}}

\newcommand{\cE}{\mathcal{E}}
\newcommand{\cF}{\mathcal{F}}

\newcommand{\cL}{\mathcal{L}}

\newcommand{\cO}{\mathcal{O}}
\newcommand{\cP}{\mathcal{P}}

\newcommand{\cR}{\mathcal{R}}
\newcommand{\cS}{\mathcal{S}}

\newcommand{\bA}{\mathbb{A}}

\newcommand{\bC}{\mathbb{C}}

\newcommand{\bF}{\mathbb{F}}

\newcommand{\bN}{\mathbb{N}}

\newcommand{\bQ}{\mathbb{Q}}
\newcommand{\bR}{\mathbb{R}}

\newcommand{\bZ}{\mathbb{Z}}

\nc{\frf}{\mathfrak{f}} 

\nc{\frs}{\mathfrak{s}}  
\nc{\frt}{\mathfrak{t}} 
\nc{\fru}{\mathfrak{u}}
  
\nc{\lsl}{\mathfrak{sl}}
\nc{\lgl}{\mathfrak{gl}}

\nc{\upsi}{\underline{\psi}}
\nc{\uchi}{\underline{\chi}}

\DeclareMathOperator{\Hom}{Hom}

\DeclareMathOperator{\GL}{GL}

\DeclareMathOperator{\sign}{\text{sign}}

\newcommand{\lra}{\longrightarrow}    

\nc{\surjto}{\twoheadrightarrow}
\nc{\ts}{\times}
\nc{\ds}{\displaystyle}
\nc{\nd}{\noindent}  
\nc{\ud}{\underline}
\nc{\ov}{\overline}
\nc{\maplra}[1]{\buildrel #1 \over \lra}
\nc{\mapto}[1]{\buildrel #1 \over \to}
\nc{\setb}[1]{\{  #1\}}

 \nc{\cHom}{\mathcal{H}om}

\newcommand{\QQ}{\mathbb{Q}}
\newcommand{\FF}{\mathbb{F}}

\newcommand{\ZZ}{\mathbb{Z}}

\newcommand{\RR}{\mathbb{R}}

\nc{\cdruur}[8] {\begin{CD} 
#1 @>#2>> #3\\ 
@AA#4A @AA#5A\\ 
#6 @>#7>> #8 
\end{CD} }
\nc{\cdrddr}[8] {\begin{CD} 
#1 @>#2>> #3\\ 
@VV#4V @VV#5V\\ 
#6 @>#7>> #8 
\end{CD} }
\nc{\dia}[8]{\xymatrix{ 
&#1 \ar@{-}[ld]_{#2} \ar@{-}[rd]^{#3} \\
#4 \ar@{-}[rd]_{#6} & &#5 \ar@{-}[ld]^{#7}\\ 
&#8} }

\nc{\diam}[9]{\xymatrix{ 
&#1 \ar@{-}[ld]_{#2}  \ar@{-}[d]^{#3} \ar@{-}[rd]^{#4} \\
#5 \ar@{-}[rd]_{#8}     & #6 \ar@{-}[d]_{#9}      & #7   \ar@{-}[ld]^{2} \\
& \bQ} } 

\nc{\sumn}[2][n]{#2_{1} +#2_{2}+ \cdots + #2_{#1}}
\nc{\poly}[3][n]{#2_{#1}#3^{#1} +#2_{#1-1}#3^{#1-1}  \cdots + #2_{1} #3+ #2_0}
\nc{\dpoly}[3][n]{#1#2_{#1}#3^{#1-1} +(#1-1)#2_{#1-1}#3^{#1-1}  \cdots +2 #2_{2} #3+ #2_1}
\nc{\mpoly}[3][n]{#3^{#1} +#2_{#1-1}#3^{#1-1}  \cdots + #2_{1} #3+ #2_0}

\nc{\vpar}[4]{    \left \{ \begin{array}{cc} #1 & \textrm{if } #2, \\
&\\
#3 & \textrm{if } #4. 
\end{array}\right. }

\nc{\ary}[5]{#1: \left\{ \begin{array}{ll} #2 &\mapsto #3 \\ #4 &\mapsto #5 \end{array} \right.}
 \nc{\bedm}{\begin{displaymath}}
 \nc{\eedm}{\end{displaymath}}
 \nc{\art}{\hbox{\bf Art}^\Z}
 \nc{\bvx}{\bos{B\!\!V}_{\! \!X}}

\newcommand{\pmat}{\left(\begin{matrix}}   
\newcommand{\epmat}{\end{matrix}\right)}   
\newcommand{\psmat}{\left(\begin{smallmatrix}}    
\newcommand{\epsmat}{\end{smallmatrix}\right)}
\nc{\twotwo}[4]{\pmat #1 & #2 \\ #3 & #4 \epmat}
\nc{\thrthr}[9]{\pmat #1 & #2 & #3 \\ #4 & #5 & #6 \\ #7 & #8 & #9 \epmat}
\nc{\stwotwo}[4]{\psmat #1 & #2 \\ #3 & #4 \epsmat}
\nc{\sthrthr}[9]{\psmat #1 & #2 & #3 \\ #4 & #5 & #6 \\ #7 & #8 & #9 \epsmat}

\def\eqalign#1{\null\,\vcenter{\openup\jot\m@th
\ialign{\strut\hfil$\displaystyle{##}$&$\displaystyle{{}##}$\hfil
\crcr#1\crcr}}\,}

\def\eqn#1#2{
\xdef #1{(\nsecsym\the\meqno)}%\writedef{#1\leftbracket#1}%
\global\advance\meqno by1
$$#2\eqno#1\eqlabeL#1
$$}

\def\a{\alpha}

\def\d{\delta}  
\def\e{\varepsilon} 
  
\def\g{\gamma}  

\def\l{\lambda}

\def\s{\sigma}

\def\R{\mathbb{R}}
\def\C{\mathbb{C}}

\def\Z{\mathbb{Z}}

\def\Q{\mathbb{Q}}

\def\mod{\hbox{ }mod\hbox{ }}

\catcode`\@=11 % This allows us to modify PLAIN macros.

\global\newcount\nsecno \global\nsecno=0
\global\newcount\meqno \global\meqno=1
\def\newsec#1{\global\advance\nsecno by1
\eqnres@t
\section{#1}}
\def\eqnres@t{\xdef\nsecsym{\the\nsecno.}\global\meqno=1}
\def\sequentialequations{\def\eqnres@t{\bigbreak}}\xdef\nsecsym{}

\def\draftmode{\message{ DRAFTMODE }
{\count255=\time\divide\count255 by 60 \xdef\hourmin{\number\count255}
\multiply\count255 by-60\advance\count255 by\time
\xdef\hourmin{\hourmin:\ifnum\count255<10 0\fi\the\count255}}}
\def\nolabels{\def\wrlabeL##1{}\def\eqlabeL##1{}\def\reflabeL##1{}}
\def\writelabels{\def\wrlabeL##1{\leavevmode\vadjust{\rlap{\smash%
{\line{{\escapechar=` \hfill\rlap{\tt\hskip.03in\string##1}}}}}}}%
\def\eqlabeL##1{{\escapechar-1\rlap{\tt\hskip.05in\string##1}}}%
\def\reflabeL##1{\noexpand\llap{\noexpand\sevenrm\string\string\string##1}
}}

\nolabels

\def\eqn#1#2{
\xdef #1{(\nsecsym\the\meqno)}%\writedef{#1\leftbracket#1}%
\global\advance\meqno by1
$$#2\eqno#1\eqlabeL#1
$$}

\def\eqalign#1{\null\,\vcenter{\openup\jot\m@th
\ialign{\strut\hfil$\displaystyle{##}$&$\displaystyle{{}##}$\hfil
\crcr#1\crcr}}\,}

\vspace{.2in}

   \nc{\hr}{[\![\hbar]\!]}
\def\foot#1{\footnote{#1}}

\catcode`\@=12 % at signs are no longer letters

\nc{\bt}{\mathbf{t}}
\draftmode

\begin{document}

%\title[The Milnor $K$-theory of the Stevens Eisenstein cocycle and the Shintani cocycle]{The Milnor $K$-theory of Stevens Eisenstein cocycle and the Shintani cocycle}
\title[The Milnor K-theory and the Shintani cocycle ]{The Milnor K-theory and the Shintani cocycle}

%\title[The Milnor K-theory and the Shintani cocycle for $\GL_n(\bQ)$ ]{The Milnor K-theory and the Shintani cocycle for $\GL_n(\bQ)$}

%\title[Milnor's K-theory, the Stevens cocycle, and the Shintani cocycle]{Milnor's K-theory, the Stevens cocycle, and the Shintani cocycle}

%\title[Stevens' Eisenstein cocycle with values in a Milnor $K$-ring and the Shintani cocycle]{Stevens' Eisenstein cocycle with values in a Milnor $K$-ring and the Shintani cocycle}

%
%\author{}
%\email{}
%\address{Department of Mathematics, POSTECH (Pohang University of Science and Technology), San 31, Hyoja-Dong, Nam-Gu, Pohang, Gyeongbuk, 790-784, South Korea. }
%%
\author{Sung Hyun Lim}
\email{finnlimsh@gmail.com}
\address{Mathematical Institute, University of Oxford, Radcliffe Observatory, Andrew Wiles Building, Woodstock Rd, Oxford OX2 6GG, United Kingdom}
\author{Jeehoon Park}
\email{jeehoonpark@postech.ac.kr}
\address{Department of Mathematics, POSTECH (Pohang University of Science and Technology), San 31, Hyoja-Dong, Nam-Gu, Pohang, Gyeongbuk, 790-784, South Korea. }

%\dedication{A dedication can be included here.}
\subjclass[2000]{11F15, 11R70}
%\subjclass[2000]{change 18G55(primary), 14J70, 14D15, 13D10 (secondary). }
%At least one subject code is required. Please refer to
%\url{http://www.ams.org/msc/} for a list of codes.

\keywords{Shintani's cocycle, Milnor's $K$-theory, modular symbols.}
%\thanks{This file documents \pkg{compositio} version \Fileversion\ and was last revised \Filedate.}

\begin{abstract}
The goal of this article is to complete the unfinished construction (due to Glenn Stevens in an old preprint \cite{S}) of a certain Milnor $K$-group valued group cocycle for $\GL_n(\bQ)$ where $n$ is a positive integer, which we call the Stevens cocycle. Moreover, we give a precise relationship between the Stevens cocycle and the Shintani cocycle, which encodes key informations on the zeta values of totally real fields of degree $n$, using the $\dlog$ map of $K$-theory and the Fourier transform of locally constant functions on $\bQ^n$ with bounded support. Roughly speaking, the Stevens cocycle is a multiplicative version of the Shintani cocyle.
\end{abstract}
\maketitle
\tableofcontents

%\vspace*{6pt}\tableofcontents  % for this guide only.
% A table of contents should normally not be included

\section{Introduction}

%I will write the introduction later after all the sections are written.

In an old unfinished preprint \cite{S} around 2007, Glenn Stevens proposed a definition of a certain Milnor $K$-group valued group cocycle for $\GL_n(\bQ)$ which is related to period integrals of Eisenstein series and the Milnor $K$-theory. Though his preprint has the crux of computations and key ideas, the definitions (of $\GL_n(\bQ)$-group actions on various objects) and theorems were not stated and proved clearly.
In fact, the second named author clarified this issue for $n=2$ case and published the result in \cite{P}. He explained how the Milnor $K_2$-group valued symbol is related to periods of Eisenstein series and $p$-adic partial zeta functions for real quadratic fields in $n=2$ case. This $n=2$ case also has intimate connections to \cite{Ce} and \cite{Sh}.

The first goal of this paper is to clarify Stevens's construction in \cite{S} for general $n \geq 2$ and construct a group cocycle for $\GL_n(\bQ)$, as is done in \cite{P} for $n=2$ case.
%Our main goal is to generalizes the work ($n=2$ case) of the second author in \cite{P} to general $n$.
The second goal is to make precise its relationship with the Shintani group cocycle for $\GL_n(\bQ)$ constructed and studied by R. Hill, \cite{H}. 
The Shintani cocycle for $\GL_n(\bQ)$ encodes key informations on the zeta values of totally real fields of degree $n$. Its construction was given in \cite{Solo} for $n=2,3$ case and later generalized to general $n$ in \cite{H}. 
%Solomon's paper and Solomon-Hu's paper.
It has an intimate relationship with the Eisenstein cocycle:
see \cite{CDG1} and \cite{CDG2} for example.
%add Dasgupta's paper, Greenberg's paper. 
The $p$-adic integrality of the Shintani cocycle and its relationship to the $p$-adic $L$-functions for totally real fields (and the $p$-adic Shintani zeta functions) was studied in \cite{AS} and \cite{Das}.
Moreover, the Shintani cocycle also plays an important role in the Gross-Stark conjecture: see \cite{Spi} for example. 
%The precise relationship between the $p$-adic $L$-functions for totally real fields and ($p$-adic integrality of) the Shintani cocycle was given in \cite{AS}

We set up the notation for statements of main theorems. Let $V=\bQ^n$ be a rational vector space of dimension $n$.
Let $V^*=\Hom(V, \bQ)$ be the dual vector space.
For any $\bQ$-algebra $R$, we define a $R$-module $V_R= V\otimes_\bQ R$.
We use $\GL(V)=\GL_n(\bQ)$ (respectively, $\GL^+(V)=\GL_n^+(\bQ)$) to denote the general linear group on $V$ (respectively, with positive determinant).
Let $\cS(V)$ be the abelian group of locally constant functions on $V$ with bounded support.
Then $\cS(V)$ is generated by characteristic function 
$$
\chi_{\underline a + d \bZ^n}
$$ 
on the affine lattice $\underline a + d \bZ^n$, where $\underline a \in \bQ^n, d \in \bN$.
For any abelian group $M$, define the space of naive distributions with values in $M$
$$
\on{Dist} (V, M) :=\Hom_{\bZ}(\cS(V), M).
$$

The first main concept is the Shintani cocycle and the Naive Shintani function. The Shintani cocycle, defined per each natural number $n$, is a homogeneous group $n$-cocycle:
    \begin{align*}
        [\Phi^{Sh}_n] \in H^{n-1}\left(\on{GL}_n(\RR), \operatorname{Dist} \left(\bA_f^n \backslash \{0\}, \bC((T_1, \cdots T_n))  \right)\right),
    \end{align*}
    where $\bA_f$ is the ring of finite adeles of $\bQ$
 and $\bC(( T_1, \cdots, T_n))$ is the field of Laurent power series with $n$-variables $T_1, \cdots, T_n$ (defined as the fraction  field of $\bC[[T_1, \cdots T_n]]$).
 We refer to \cite{H} for its concrete relationship with the zeta values of a totally real field of degree $n$.
 The Shintani cocycle $\Phi_n^{Sh}$ is defined by the Solomon-Hu pairing (see \ref{SH} for its description and also see \cite{Solo})
\begin{align*}
& \langle \cdot , \cdot \rangle_{SH} : \cL_{\QQ^n} \times \cS(\bA_f^n \backslash \{0\}) \rightarrow \bC((T_1, \cdots T_n )),
\end{align*}
where $\mathcal L_{\Q^n}$ is the abelian group generated by characteristic functions of open rational cones, modulo constant functions,
 and an appropriate polyhedral cone function $\sigma_n^{Sh}$:
$$\Phi_n^{Sh}:= \langle \sigma_n^{Sh}, \cdot \rangle_{SH},$$
The polyhedral cone function $\sigma_n^{Sh}$ (see \ref{pcf}) is defined as a cone function in $\FF^n$ for ordered field $\FF = \RR((\epsilon_1)) \cdots ((\epsilon_n))$ followed by the restriction to $\R^n$ under the embedding $\R^n \rightarrow \FF^n$.
%where $c=\chi_{\bR^+ v_1 + \cdots + \bR^+ v_r}$ and vectors $\{ v_1, \cdots v_r \} \subset \Q^n$ are assumed to belong to the lattice of periods of test function $f$ (if $L_f$ is the lattice of periods of $f$, by linear dependence $\forall v \in \QQ^n, \exists n: n \cdot v \in L_f$) and $\cP_{c,f} := (0,1] v_1 + \cdots + (0,1] v_r$ is the fundamental parallelogram for $c$. 
 We define a related variant of the Shintani cocycle, which we call \textit{the Naive Shintani function}:
 
\begin{defn}\label{nsh}
Define a function $\Phi_n^{NSh}$ %which we call the Naive Shintani function:
\begin{eqnarray*}
&\Phi_n^{NSh}  : (\GL_n \Q)^n \rightarrow \on{Dist} ( \Q^n, \bC (( T_1, \cdots T_n )) ) \\
 &\Phi_n^{NSh}(\gamma_1, \cdots \gamma_n)(f) := \langle \chi_{\RR^+ e_1 \gamma_1^{T} + \cdots + \RR^+ e_1 \gamma_n^{T} }, f \rangle_{SH}.
\end{eqnarray*}
%so that 
%$$
%  \Phi_n^{Sh} (1, \rho, \cdots \rho^{n-1}) =   \Phi_n^{NSh} (1, \rho, \cdots \rho^{n-1}).
%$$
%Here, $\bC (( T_1, \cdots T_n))$ is defined as the fraction field of $\bC [[T_1, \cdots T_n ]]$. 
\end{defn}
It shares similar properties with the Shintani cocycle:
$$
  \Phi_n^{Sh} (1, \rho, \cdots \rho^{n-1}) = \langle \chi_{\RR^+e_1+ \cdots + \RR^+e_n}, f \rangle_{SH}= \Phi_n^{NSh} (1, \rho, \cdots \rho^{n-1}), \  \rho^{-1}:=\begin{pmatrix} 0 & 1 & 0 &\cdots& \cdots & 0 \\ 0& 0& 1 & 0 & \cdots  & 0  \\ 0& \cdots& \cdots &\cdots& \cdots& 0\\ 0&0 &\cdots& \cdots& 1& 0\\ 1&0 &\cdots &\cdots &\cdots & 0\end{pmatrix}
$$
where the matrix $\rho$ is the shift-permutation matrix defined by $\rho_{1,n} = 1, \rho_{i+1, i} = 1$ for $i=1, \cdots n-1$ and all other entries zero, and under the relevant group actions (see the appendix I, subsection \ref{sec4.1} for our precise description of the group actions) we have
    \begin{align*}
        \Phi_n^{Sh}(\gamma \gamma_1 ,\cdots \gamma \gamma_n) = (\sign \g) \Phi_n^{Sh}(\gamma_1, \cdots \gamma_n) |_{\gamma^{T}}, 
        \qquad   \Phi_n^{NSh}(\gamma \gamma_1 ,\cdots \gamma \gamma_n) = \Phi_n^{NSh}(\gamma_1, \cdots \gamma_n) |_{\gamma^{T}},
    \end{align*}
    for $\g, \g_1, \cdots, \g_n \in \GL_n(\bQ)$.
See Proposition \ref{rhos} for the first statement and Proposition \ref{Sproperty} for the second statement (the $\GL_n(\bQ)$-equivariance) for $\Phi_n^{Sh}$ and $\Phi_n^{NSh}$. Observe that the Naive Shintani function  and  the Shintani cocycle are both defined using the Solomon-Hu pairing, but the Naive version has a simpler  definition involving cone functions  with basis  vectors  moved simply by matrices. This is indeed the most naive way to proceed, but is not enough  to make  the Naive Shintani function  into a group cocycle; much work has been done, for example in \cite{H}, to modify the cone function to fix the naive definition. Part of our result (Corollary \ref{co}) is that the Naive Shintani function is a cocycle in an appropriate  quotient group.

Next we introduce the Stevens cocycle, whose definition involves Milnor K-theory. Let $K_n^M(V)$ be the  $n$-th Milnor $K$-group of the ring of trigonometric functions (see \ref{crv}).
There is an additive homomorphism (see \ref{dlo} for its precise definition and also look at section 2.3, \cite{K} for its theoretical description):
\begin{eqnarray*}
{\dlog}: & K_n^M(V) & \rightarrow \mathbb C((z_1, \cdots z_n)) \mathrm{d} z_1 \wedge \cdots \wedge \mathrm{d} z_n \\
& \{f_1, \cdots, f_n\} & \mapsto \dlog f_1 \wedge \cdots \wedge \dlog f_n.
\end{eqnarray*}
Let $\tilde{K}_n^M(V)=K_n^M(V)/J_n'(V)$ is a certain quotient of $K_n^M(V)$ (see Definition \ref{targetk}) and 
$$J_n' \subseteq  \mathbb C((z_1, \cdots z_n)) \mathrm{d} z_1 \wedge \cdots \wedge \mathrm{d} z_n
=\mathbb C((T_1, \cdots T_n)) \mathrm{d} T_1 \wedge \cdots \wedge \mathrm{d} T_n$$
 be the image of $J_n'(V)$ under the map ${\dlog}$,
where we identify $T_j = 2\pi i z_j$, $j=1, \cdots, n$.
Now we state our main theorem.
 \begin{thm}
There exists a homogeneous group $n$-cocycle $\Phi_n^{St}$
\begin{eqnarray*}
[\Phi_n^{St}] \in  H^{n-1}\left(\GL_n (\Q), \on{Dist} ( \Q^n, \tilde{K}_n^M(V) )\right)
\end{eqnarray*}
%where $\tilde{K}_n^M(V)$ is defined in Definition \ref{targetk}. 
with the following property:
for given $(\gamma_1, \cdots, \gamma_n) \in \GL_n (\Q)^n$ and $f \in \cS (\Q^n)$, we have
\begin{eqnarray*}
{\dlog}( \Phi_n^{St}(\gamma_1, \cdots, \gamma_n)(f)) \equiv (-1)^n \cdot (\Phi_n^{NSh} (\gamma_1, \cdots, \gamma_n)(\hat f)) \mathrm{d} T_1 \wedge \cdots \wedge \mathrm{d} T_n  \quad (\operatorname{mod} J_n'),
\end{eqnarray*}
 %under the identification $T_j = 2\pi i z_j$, $j=1, \cdots, n$,
 %where $J_n'$ is given in \ref{jnp} and 
 where $\hat{f}$ is the Fourier transform of $f$.
\end{thm}

See the appendix, subsection \ref{sec4.2} for the definition of $\hat{f}$. The key observation for the above theorem is to use the Fourier transform $\hat{f}$ on the right hand side for comparison. We refer to Theorem \ref{mt} and Theorem \ref{EisNSh} for further details and their proofs.
%Let $\tilde G$ be the \textcolor{blue}{subgroup(???)} consisting of $(\gamma_1, \cdots, \gamma_n) \in \GL_n (\Q)^n$ such that the first columns of $\gamma_1, \cdots, \gamma_n$ form an orthogonal matrix. 
The above theorem implies the following corollary.
\begin{cor} \label{co}
The Naive Shintani function $\Phi_n^{NSh} \operatorname{mod} J_n'$ becomes a cocycle, i.e.
$$
[\Phi_n^{NSh} \operatorname{mod} J_n'] \in H^{n-1}\left(\GL_n(\Q), \on{Dist} (\bQ^n, \frac{\bC((T_1, \cdots, T_n))dT_1 \wedge \cdots dT_n}{J_n'})\right).
$$
\end{cor}

\vspace{2em}

Now we briefly explain contents of the article.
Section \ref{sec2} is devoted to the construction of the Stevens cocycle. In subsection \ref{sec2.1}, we introduce the ring $\cR(V)$ of trigonometric functions and the Milnor $K$-ring $K^M(V)$. In subsection \ref{sec2.2}, we define an explicit map from the cohomology associated to simplexes (sometimes called the cohomology with compact support) to group cohomology for $\GL_n(\bQ)$. Then we give a construction for the multiplicative Kubota-Leopoldt distribution ($n=1$ case), which is a basis for the Stevens cocycle in subsection \ref{sec2.3}.
In subsection \ref{sec2.4}, we construct the Stevens cocycle using the materials developed so far.

Section \ref{sec3} is devoted to the comparison between the Shintani cocycle and the Stevens cocycle. In subsection \ref{sec3.1}, we briefly review the Shintani cocycle. In subsection \ref{sec3.2}, we propose the definition of the Naive Shintani function. Then we give the precise comparison in subsection \ref{sec3.3}.

In the appendix, we summarize our convention of group actions on various objects and the Fourier theory for locally constant functions with bounded support in section \ref{sec4}.

\subsection{Acknowledgement}
The work of JP was partially supported by BRL (Basic Research Lab) through the NRF (National Research Foundation) of South Korea (NRF-2018R1A4A1023590). Authors would like to thank Glenn Stevens whose preprint is a starting point of this project and its influence on our article is obvious.

\section{The Milnor $K$-group and the Stevens group cocycle} \label{sec2}

\subsection{The Milnor $K$-group of trigonometric functions}\label{sec2.1}

For $w=(r, \lambda) \in \Q \times V^*$, we consider a trigonometric function $\e_w(z) := \exp(2\pi i ( \l (z) - r)), z \in \bC^n$, defined on $V_\bC = \bC^n$.
%Let $\cO_\bQ(V_\bC)$ be the ring of holomorphic functions on $V_\bC\setminus V$.
Let $\cO_{\bQ}(V_\bC)$ be the ring of meromorphic functions on $V_\bC$ with possible poles only on $V=\bQ^n$.
We define a subgroup $\cE(V)$ (respectively, $\cC(V))$ of the unit group $\cO_{\bQ}(V_\bC)^*$ generated
by $\e_w(z)$ (respectively, $\e_w(z), 1-\e_w(z)$):
\begin{eqnarray*}
\cE(V) &=& \langle \e_w(z) : w=(r, \lambda) \in \Q \times V^*, \l\neq 0 \rangle, \\
\cC(V) &=& \langle \e_w(z), 1-\e_w(z) : w=(r, \lambda) \in \Q \times V^*, \l\neq 0 \rangle. 
\end{eqnarray*}
Define the subring $\cR(V)$ of $\cO_{\bQ}(V_\bC)$:
\begin{eqnarray}\label{crv}
\begin{split}
\cR(V) &:= \bZ[ \e_w(z), 1-\e_w(z) : w=(r, \lambda) \in \Q \times V^*, \l\neq 0] \\
& =  \{\sum_{j=1}^m m_j \e_{w_j}(z_j) : w_j = (r_j, \lambda_j) \in \Q \times V^*, m_j \in \bZ \}.
\end{split}
\end{eqnarray}
Note that $\cR(V)$ is a commutative ring with unity $1$.

Let $A$ be a commutative ring with 1. The Milnor $K$-groups are defined as follows:
\begin{eqnarray*}
K_0^M(A) &=& \bZ, \quad K_1^M(A) = A^* \\
K_2^M(A) &=& A^* \otimes A^* / \langle a\otimes b :  a+b = 0 \text{ or } 1 \rangle \\
K_n^M(A) &=& A^* \otimes \cdots \otimes A^* / \langle a_1 \otimes \cdots \otimes a_n
 : \text{ there exists } i, j \text { such that } a_i+a_j= 0 \text{ or } 1 \rangle \\
 K^M(A) &=& \bigoplus_{n\geq 0} K_n^M(A).
\end{eqnarray*}
Thus, $K^M_n(A)$ is an abelian group and $K^M(A)$ is the graded $\bZ$-algebra generated by $K_1^M(A)=A^*$ modulo the Steinberg relation ($a+b=0$ or $1$) and we use the notation $\{x_1, \cdots, x_n\} \in K_n^M(A)$ to denote the element $x_1\otimes \cdots \otimes x_n$ modulo the Steinberg relation:
$$
K_n^M(A) \otimes K_m^M(A) \to K_{n+m}^M(A), \quad \{x_1, \cdots, x_n\} \cdot \{y_1, \cdots, y_m\} := \{x_1, \cdots, x_n, y_1, \cdots, y_m\}.
$$

For the abelian group structure of $K_n^M(A)$, we have the following multi-linearity by definition:
$$
\{x_1, x_2, \cdots, x_n\} +\{x_1', x_2, \cdots, x_n  \}=\{x_1 \cdot x_1', x_2, \cdots, x_n\}.
$$
The similar multi-linearity holds for other components. 
A straightforward computation shows that 
\begin{eqnarray*}
&&\{x,y\} = -\{y, x\}, \quad x,y \in A^*, \\
&&\xi \cdot\eta = (-1)^{nm} \eta \cdot \xi, \quad \xi \in K_n^M(A), \eta \in K_m^M(A), \\
&& \{x,x\}= \{x, -1\} =\{-1, x\}, \quad  \{-x^{-1}, x\} =0,  \quad x \in A^*.
\end{eqnarray*}

For $n\geq 1$, let $J_n(V)$ be the subgroup of $K_n^M(V):=K_n^M(\cR(V))$ generated by elements of the form
$\{x_1, \cdots, x_n\}$ with $x_i \in \cC(V)$ and at least one $x_i \in \cE(V)$. Then
$J(V)=\bigoplus_{n\geq 0} J_n(V)$ becomes an ideal of $K^M(V):=K^M(\cR(V))$.
\begin{defn}\label{targetk} 
Let $\tilde J(V)$ be the ideal of $K^M(V)$ generated by $J(V)$ and $\{-1\}$.\foot{Note that $-1 \notin \cE(Q)$ and so $\{-1\} \notin \cE(V)$.}
%Let $\tilde J$ be the ideal of $K^M(V)$ generated by $J$ and $\{-1,x_1, \cdots, x_m\}, m\geq 0, x_i \in A^\times$.
Put $\tilde J_n(V) := \tilde J(V) \cap K_n^M(V)$ be the subgroup of $K_n^M(V)$.
Define
\begin{eqnarray*}
\tilde K_n^M (V) := K_n^M(V)/\tilde J_n(V), \quad
\tilde K^M(V) := K^M(V)/\tilde J(V). \nonumber
\end{eqnarray*}
%\begin{eqnarray*}
%\tilde K^n_M (V) &:=& K^n_M(V)/\tilde J_n, \\
%\tilde K_M(V) &:=& K_M(V)/\tilde J. \nonumber
%\end{eqnarray*}
\end{defn}

Stevens showed the following theorem, called the Dedekind reciprocity law.
\begin{thm}\label{drl}
Let $A$ be an arbitrary commutative ring. Let $n\geq 1$  and  $u_1, \cdots, u_n \in A^\times$. Let $u_0=u_1+\cdots + u_n$.
Suppose for all $k = 1, \cdots , n $ that $u_1 + \cdots + u_k \in A^\times$. 
Then 
$$
\sum_{i=0}^n(-1)^i \{u_1, \cdots, \hat{u}_i, \cdots, u_n \}
$$
%lies in the ideal generated by $\{ -1, x_1, \cdots, x_m\}$ for all integers $m \geq 0$ and all  $x_i \in A^\times$.
lies in the ideal $I$ generated by $\{ -1\}$.
\end{thm}
\begin{proof}
We include Stevens' proof for reader's convenience.
The proof goes by induction on $n$. The case $n=1$ is obvious.
When $n=2$, $\frac{u_1}{u_0}+\frac{u_2}{u_0}=1$. Thus
\begin{eqnarray*}
0&=&\{ \frac{u_1}{u_0}, \frac{u_2}{u_0}\} \\
&=&\{u_1, u_2 \}-\{u_0, u_2 \}+\{u_0, u_1 \}+\{u_0, u_0 \}\\
&=&\{u_1, u_2 \}-\{u_0, u_2 \}+\{u_0, u_1 \}+\{-1, u_0 \},
\end{eqnarray*}
which proves the $n=2$ assertion.
Now suppose $m\geq 2$ and the result is true for $n = m$. 

We will prove the result for $n=m+1$. Let $v_1, \cdots, v_{m+1} \in A^\times$ and set
$v_0=v_1+\cdots+v_{m+1}$. Also, let $u_0=u_1+\cdots+ u_m$ with $u_1=v_1, \cdots, u_m=v_m$. Then by the induction hypothesis
\begin{eqnarray}\label{ih}
\sum_{i=0}^m (-1)^i \left\{u_0, \cdots, \hat{u}_i, \cdots, u_m \right\} \in I.
\end{eqnarray}
Multiplying \ref{ih} on the left by $\{v_{m+1}\}$ and $\{v_0\}$ we get
\begin{eqnarray*}
A:=\sum_{i=0}^m (-1)^i \left\{v_{m+1}, u_0, \cdots, \hat{u}_i, \cdots, u_m \right\} &\in& I,\\
B:=\sum_{i=0}^m (-1)^i \left\{v_0, u_0, \cdots, \hat{u}_i, \cdots, u_m \right\} &\in& I.
\end{eqnarray*}
By using the fact $v_0=u_0+v_{m+1}$ we obtain
$$
\{u_0, v_{m+1}\} - \{v_0, v_{m+1} \} + \{v_0, u_0\} \in I.
$$
Therefore we have
$$
A-B=\left\{v_{m+1}, u_1, \cdots, u_m \right\} + \sum_{i=0}^m (-1)^i \left\{v_{m+1}, v_0, \cdots, \hat{u}_i, \cdots, u_m \right\} - \left\{v_0, u_1, \cdots, u_m \right\} \in I,
$$
which means (after multiplying $(-1)^m$) that
$$
\sum_{i=0}^{m+1} (-1)^i \left\{v_0, \cdots, \hat{v}_i, \cdots, v_{m+1} \right\} \in I.
$$
This proves the assertion for $n=m+1$. The theorem follows by induction.
\end{proof}

\subsection{Cohomology associated to simplexes and group cohomology for $\GL_n(\bQ)$} \label{sec2.2}

%A strongly convex rational polyhedral cone $\sigma \subseteq V_\bR$ is called simplicial
%if its minimal generators are linearly independent over $\bR$.

Let $V=\bQ^n$. For $k \geq 0$, let $(V^*)^{k+1}_{ng}$ be the subset of $(V^*)^{k+1}$ consisting of vectors $\l_0, \cdots,
\l_{k}$ such that every $m$-element (with $m \leq n$) subset of $\{ \l_0, \cdots,\l_{k}\}$ is linearly independent.
%The group $(\bQ^\times)^{k+1}$ acts componentwise on $(V^*)^{k+1}$.
%The notation $(V^*)^{k+1}_{ng}/ (\bQ^\times)^{k+1}$ means the set of $(\bQ^\times)^{k+1}$-orbits in $(V^*)^{k+1}_{ng}$. For $\l \in (V^*)^{k+1}_{ng}$ we let $[\l]$ denote the $(\bQ^\times)^{k+1}$-orbit of $\l$. For each integer $k\geq 0$, we define
%\begin{eqnarray*}
%C_k(V):=\text{the free abelian group generated by } (V^*)^{k+1}_{ng}/ (\bQ^\times)^{k+1}, \quad
%C(V):=\bigoplus_{m\geq 0} C_k(V).
%\end{eqnarray*}
The group $(\bQ^+)^{k+1}$ acts componentwise on $(V^*)^{k+1}$.
The notation $(V^*)^{k+1}_{ng}/ (\bQ^+)^{k+1}$ means the set of $(\bQ^+)^{k+1}$-orbits in $(V^*)^{k+1}_{ng}$. 
The $(\bQ^+)^{k+1}$-orbit in $(V^*)^{k+1}_{ng}$ can viewed as 
a $k$-simplex, i.e. a $k$-dimensional polytope in $V_\bR$ with $k+1$ vertices.
For $\l \in (V^*)^{k+1}_{ng}$ we let $[\l]$ denote the $(\bQ^+)^{k+1}$-orbit of $\l$. 
Put $C_{-1}(V)=\bZ$. For each integer $k\geq 0$, we define
\begin{eqnarray*}
C_k(V):=\text{the free abelian group generated by } (V^*)^{k+1}_{ng}/ (\bQ^+)^{k+1}, \quad
C(V)=C_\bullet(V):=\bigoplus_{m\geq -1} C_k(V).
\end{eqnarray*}

If we define the boundary map $\partial: C_k(V) \to C_{k-1}(V), k\geq 1$,
$$
\partial ([\l_0, \cdots, \l_k])=\sum_{i=0}^n(-1)^i [\l_0, \cdots,\hat{\l}_i\cdots. \l_k],
$$
and $\partial(C_0(V))=0$, then $\partial^2=0$. 
%\textcolor{blue}{Not $\partial|_{C_0(V)} = \text{deg}$?}
The standard computation shows that
\[
\xymatrix{
\cdots \xrightarrow{\partial} C_k(V) \xrightarrow{\partial} C_{k-1}(V) \xrightarrow{\partial} \cdots\xrightarrow{\partial} C_{0}(V) \xrightarrow{\text{deg}} \bZ \to 0
}
\]
is exact.
The group $\GL(V)$ acts on $V^*$ by $(\g \cdot \l)(x):= \l(x \cdot \g)$
for $\l \in V^*, x \in V, \g \in \GL(V)$.
This induces a left $\GL(V)$-action on $(V^*)^{k+1}_{ng}/ (\bQ^+)^{k+1}$:
$$
\g \cdot [\l_0, \cdots, \l_k] = [\g \cdot \l_0, \cdots, \g \cdot \l_k], 
\quad \g \in \GL(V), [\l_0, \cdots, \l_k]  \in C_k(V),
$$
since $\GL(V)$-action on $(V^*)^{k+1}_{ng}$ commutes with the $(\bQ^+)^{k+1}$-action. Consequently, this induces a $\bZ$-linear left
$\GL(V)$-action on $C_k(V)$.

For any right $\bZ[\GL(V)]$-module $M$, we consider the following right $\GL(V)$-action 
on  $\Hom(C_\bullet(V), M)$
by 
$$
\left(\xi |_\g\right) ([\l_0, \cdots, \l_k]) =\xi ([\g \cdot \l_0, \cdots, \g\cdot\l_k])|_\g, \quad \xi \in \Hom(C_\bullet(V), M), \g \in \GL(V).
$$
Let us define the cochain complex
$$ 
C^\bullet(V, M) := \Hom_{\GL(V)}(C_\bullet(V), M):=\{ \xi \in \Hom(C_\bullet(V), M) : \xi|_\g =\xi, \g \in \GL(V)  \}
$$
with the induced coboundary operator $\partial^*$ from $(C_\bullet(V), \partial)$. Let $Z^k(V,M)$ (respectvely, $B^k(V,M)$) be the submodule of $C^k(V,M)$ consisting of cocycles (respectively, coboundaries).
Define
$$
H^i(V,M) = H^i (C^\bullet (V,M)):=Z^i(V,M)/B^i(V,M),  \quad i \geq 0.
$$

%\textcolor{blue}{Why include this?} Note that  $C_\bullet(V)[-1]$, which is degree -1 shift\footnote{Here $C_i(V)[-1]=C_{i-1}(V)$
%for every $i \geq 0$.
% Thus $C_\bullet(V)[-1]=\bigoplus_{m\geq 0} C_{m}(V)[-1]$.} of $C_\bullet(V)$, is equipped with a $\bZ$-graded associative algebra under the multiplication
%$$
%C_m(V)[-1] \times C_n(V)[-1] \to C_{m+n}(V)[-1], \quad [\l_0, \cdots, \l_{m}] 
%\cdot [\m_0, \cdots, \m_{n}]:= [\l_0, \cdots, \l_{m}, \m_0, \cdots, \m_{n}],
%$$
%
%with the convention that $ [\l_0, \cdots, \l_{m}, \m_0, \cdots, \m_{n}]=0$ when
%$ (\l_0, \cdots, \l_{m}, \m_0, \cdots, \m_{n})$ does not belong to $(V^*)^{m+n+2}_{ng}$.

For any right $\bZ[\GL(V)]$-module $M$, let $C^\bullet (\GL(V),M)$ be the standard homogeneous group cochain complex on $\GL(V)$ with values in $M$.
So the element in $C^k(\GL(V),M)$ is a function $f: \GL(V)^{k+1} \to M$ such that
$$
\quad f|_\g=f, \text{ where } (f|_\g) (\sigma_0, \cdots, \sigma_k):= f (\g \cdot \sigma_0, \cdots, \g \cdot \sigma_k))|_\g, \quad \sigma_i, \g \in \GL(V).
$$
Recall the coboundary map $\delta: C^{k-1}(\GL(V),M) \to C^k(\GL(V),M)$ satisfying $\d^2=0$:
$$
\d(f)  (\sigma_0, \cdots, \sigma_k)=\sum_{i=0}^k (-1)^i f (\sigma_0,\cdots,
\hat{\sigma}_i, \cdots, \sigma_k).
$$
Then the group cohomology $H^i(\GL(V), M)$ can be described as 
%\textcolor{blue}{($\simeq$ not $\cong$?)}
$$
H^i(\GL(V),M) \cong H^i (C^\bullet (\GL(V),M)).
$$
Now we examine a relationship between $C^\bullet(V,M)$ and $C^\bullet(\GL(V),M)$.
Choose a covector $\l \in V^*$.
The assignment 
$$
\Phi_\l (\xi) (\g_0, \cdots, \g_k) := \xi([\g_0 \cdot \l, \cdots, \g_{k} \cdot \l]),
$$
where $\xi \in C^n(V, M),  (\g_0, \cdots, \g_k) \in \GL(V)^{k+1}$,
induces a well-defined homomorphism from $C^k(V, M)$ to $C^k(\GL(V), M)$, because
$$
\xi|_\g =\xi \text{ implies that  } \Phi(\xi)_\l |_\g = \Phi(\xi)_\l .
$$

\begin{prop}
For any choice $\l \in V^*$, the map $\xi \mapsto \Phi_\l (\xi)$ is a cochain
map from $C^\bullet(V, M)$ to $C^\bullet(\GL(V), M)$. If we change $\l$ by
$\l'=\sigma \l$ for some $\sigma\in \GL(V)$, then 
\begin{eqnarray}\label{want}
\Phi_\l (\xi)(\g_0, \g_1, \cdots, \g_{k-1}) - \Phi_{\l'}(\xi)(\g_0, \g_1,
\cdots, \g_{k-1}) = (\d C_\sigma) (\g_0, \cdots, \g_{k-1}), \quad k \geq 2,
\end{eqnarray}
for some $C_\sigma \in C^{k-2}(\GL(V), M)$. 
If $\xi$ is a cocycle, then $\Phi_\l(\xi) (\g_0)-\Phi_{\l'} (\xi)(\g_0)=0$.
%For $k\geq 1$, $\Phi(\xi)_\l (\g_0)-\Phi(\xi)_{\s \l} (\g_0)= m - m|_\sigma$ (?)for some $m \in M$.
\end{prop}

\begin{proof}
The cochain property $\Phi_\l (\partial (\xi))= \d (\Phi_{\l}(\xi))$ follows from a straightforward computation.

If we define $C_\sigma \in C^{k-2}(\GL(V), M), k \geq 2$, by the formula
\begin{eqnarray*}
C_\sigma (\g_1, \cdots, \g_{k-1})
:=\sum_{i=2}^k(-1)^i \xi ([\g_1 \cdot\l, \cdots, \g_{i-1}\cdot\l,  (\g_{i}\s )\cdot \l, (\g_{i+1}\sigma)\cdot  \l,\cdots,  (\g_{k-1}\sigma) \cdot \l ]),
\end{eqnarray*}
a direct computation shows \ref{want}. In the case of $k=1$, we have
$$
\Phi_{\l}(\xi)(\g_0)-\Phi_{\sigma \l}(\xi)(\g_0) = \xi([\g_0 \l]) - \xi([\g_0 \sigma \l]) =0,
%\xi[\l]|_{\g_0^{-1}} - \xi[\l]|_{\s^{-1} \g_0^{-1}} =0,
$$
since $\xi$ is a cocycle, i.e. $\xi \in Z^0(V,M):=\{f \in C^0(V,M) : f \text{ is a constant function on } C_0(V)  \}$.
\end{proof}

\begin{cor}\label{GpCocycle}
For $m \geq 0$, the assignment $\xi \mapsto \Phi_\l(\xi)$ induces a homomorphism $Z^m(V,M) \to Z^m(\GL(V),M)$ and consequently a homomorphism $\Phi=\Phi^m: H^m(V,M) \to H^m(\GL(V),M)$, which is independent of $\l \in V^*$. 
\end{cor}
%For $m=0$, we can define an identity map 
%$\Phi^0=Id: H^0(V,M) \to H^0(\GL(V),M)$, because both $H^0(V,M)$ and $H^0(\GL(V),M)$ are canonically isomorphic to $M^{\GL(V)}$.

%For $k=0$, there is a homomorphism $\Phi_\l :H^0(V,M) \to H^0(\GL(V),M)$, depending on $\l \in V^*$.

\subsection{The multiplicative Kubota-Leopoldt distribution} \label{sec2.3}

We define $\eta\in \operatorname{Dist}(\bQ, \cC(\bQ))$:=$\Hom(\cS(\bQ), \cC(\bQ))$:
\begin{eqnarray}
\eta \left(\sum_{i=1}^m\a_i \cdot  \chi_{a_i + d_i \bZ}\right) =\prod_{i=1}^m \left( 1 - e(\frac{z-a_i}{d_i})\right)^{\a_i}, \quad z \in \bC,
\end{eqnarray}
where $e(z) = e^{2 \pi i z}$ and $d_i, a_i \in \bQ^+, a_i < d_i, \a_i \in \bZ$.
We need to show that the definition of $\xi$ does not depend on a presentation of a test function $f \in \cS(\bQ)$.
If we write
$$
{a} + d\bZ^ =
\bigcup_{b=0}^{m-1} {a} + d( {b} +m\bZ),
$$
then
\begin{eqnarray*}
\chi_{{a}+d\bZ} = \sum_{b=0}^{m-1}\chi_{{a} + d {b} + dm \bZ}.
\end{eqnarray*}

The following computation shows that $\eta$ is well-defined:
\begin{eqnarray} \label{wd}
\eta(\sum_{b=0}^{m-1} \nonumber
\chi_{{a} + d {b} + dm \bZ}) &=&\prod_{b =0}^{m-1} \left( 1 - e(\frac{z-a-db}{dm})\right)\\ %\nonumber
&=& (1-X)(1-X \zeta_m) \cdots (1-X \zeta_m^{m-1})
 \text{ where } X=e(\frac{z-a}{dm}), \zeta_m=e(\frac{-1}{m}) \\ 
 &=& (1-X^m) = 1-e(\frac{z-a}{d}) =\eta(\chi_{a+d \bZ})(z). \nonumber
\end{eqnarray}

\begin{prop}
The map $\eta$ is a group homomorphism from $(\cS(\bQ), +)$ to $(\cC(\bQ), \cdot)$ such that 
$$
\eta|_\g = \eta, \quad \g \in \GL_1^+(\bQ)=\bQ^+.
$$
\end{prop}

\begin{proof}
By construction, $\eta$ is a group homomorphism. For $a, d\in \bQ^+, a < d, \g \in \GL_1^+(\bQ)$, we compute
\begin{eqnarray*}
\left( \eta|_\g \right) (\chi_{a+d\bZ}) &=& \eta (\g \cdot \chi_{a+ d\bZ})|_\g \\
&=& \eta (\chi_{a \cdot \g^{-1} + d\bZ \cdot \g^{-1}})|_{\g}  \\
&=& \left( 1 - e\left(\frac{z - a \cdot \g^{-1}}{|d \cdot \g^{-1}|}\right) \right)|_\g \\
&=&1 - e\left(\frac{z\cdot \g^{-1} - a \cdot \g^{-1}}{d \cdot \g^{-1}}\right) = 1- e\left(\frac{z-a}{d}\right)= \eta(\chi_{a+d\bZ}).
\end{eqnarray*}
This proves the claim.
\end{proof}

When $\g \in \GL_1(\bQ)$ is a negative rational number, we have that 
$$
(\eta|_\g) (\chi_{a+d \bZ})= - e \left(\frac{-z+a}{d}\right) \cdot \eta(\chi_{a+d\bZ}),
$$
for $a, d \in \bQ^+, a < d$.

%{\color{blue} This blue part contains an error. Will be removed later.
%We can define a variant of the multiplicative Kubota-Leopoldt distribution which has more
%symmetric behavior with respect to the sign of $\g \in \GL_1(\bQ)$:
%\begin{eqnarray}
%\eta^+ \left(\sum_{i=1}^m\a_i \cdot  \chi_{a_i + d_i \bZ}\right) :=
%\prod_{i=1}^m \sqrt{-1}^{\a_i}\left( e(\frac{-z+a_i}{2d_i}) - e(\frac{z-a_i}{2d_i})\right)^{\a_i}, \quad z \in \bC,
%\end{eqnarray}
%where $e(z) = e^{2 \pi i z}$ and $d_i, a_i \in \bQ^+, a_i < d_i, \a_i \in \bZ$.
%Then
%$$
%\eta^+(\chi_\bZ)=  \frac{e^{-\pi i z}-e^{\pi i z}}{- i}=2\sin(\pi z).
%%- \frac{e^{-\pi \sqrt{-1} z}-e^{\pi \sqrt{-1} z}}{-2 \sqrt{-1}}=\sin (\pi z).
%$$
%A simple computation confirms the following proposition.
%\begin{prop} (seems to be wrong)
%The map $\eta^+$ is a group homomorphism from $(\cS(\bQ), +)$ to $(\cC(\bQ), \cdot)$ such that 
%$$
%\eta^+|_\g =\sgn(\g) \eta^+, \quad \g \in \GL_1(\bQ)=\bQ^\times.
%$$
%\end{prop}
%}

We can define a variant of the multiplicative Kubota-Leopoldt distribution which is invariant under the action of $\g \in \GL_1(\bQ)$:
\begin{eqnarray}
\eta^+ \left(\sum_{i=1}^m\a_i \cdot  \chi_{a_i + d_i \bZ}\right) :=
\prod_{i=1}^m \left( e(\frac{z-a_i}{2d_i}) - e(\frac{-z+a_i}{2d_i})\right)^{\a_i}, \quad z \in \bC,
\end{eqnarray}
where $e(z) = e^{2 \pi i z}$ and $d_i, a_i \in \bQ^+, a_i < d_i, \a_i \in \bZ$.
Then
$$
\eta^+(\chi_\bZ)=  e^{\pi i z}-e^{-\pi i z}=2i\sin(\pi z).
$$

\begin{prop} 
The map $\eta^+$ is a group homomorphism from $(\cS(\bQ), +)$ to $(\cC(\bQ), \cdot)$ such that 
$$
\eta^+|_\g = \eta^+, \quad \g \in \GL_1(\bQ)=\bQ^\times.
$$
\end{prop}
\begin{proof}
When $\g \in \GL_1(\bQ)$ is a negative rational number,
\begin{eqnarray*}
\left( \eta|_\g \right) (\chi_{a+d\bZ}) &=& \eta (\g \cdot \chi_{a+ d\bZ})|_\g \\
&=& \eta (\chi_{a \cdot \g^{-1} + d\bZ \cdot \g^{-1}})|_{\g}  \\
&=& \eta (\chi_{a \cdot \g^{-1} +|d \cdot \g^{-1}|+ |d \cdot \g^{-1}|\bZ})|_{\g}  \\
&=& \left( e\left(\frac{z - a \cdot \g^{-1}-|d \cdot \g^{-1}|}{2|d \cdot \g^{-1}|}\right)-e\left(\frac{-z + a \cdot \g^{-1}+|d \cdot \g^{-1}|}{2|d \cdot \g^{-1}|}\right) \right)|_\g \\
&=& \left( e\left(\frac{z \cdot \g^{-1} - a \cdot \g^{-1}-|d \cdot \g^{-1}|}{2|d \cdot \g^{-1}|}\right)-e\left(\frac{-z \cdot \g^{-1} + a \cdot \g^{-1}+|d \cdot \g^{-1}|}{2|d \cdot \g^{-1}|}\right) \right) \\
&=& -e\left(\frac{z \cdot \g^{-1} - a \cdot \g^{-1}}{-2d \cdot \g^{-1}}\right)+e\left(\frac{-z \cdot \g^{-1} + a \cdot \g^{-1}}{-2d \cdot \g^{-1}}\right)  \\
&=&= e\left(\frac{z-a}{2d}\right)-e\left(\frac{-z+a}{2d}\right)= \eta(\chi_{a+d\bZ}).
\end{eqnarray*}
When $\g \in \GL_1(\bQ)$ has a positive sign, the computation is trivial.
\end{proof}

\subsection{Cocycle with values in a Milnor $K$-group valued distribution}\label{sec2.4}

Let $\{e_1^*, \cdots, e_n^*\}$ be the dual basis to the standard basis $\{e_1, \cdots, e_n\}$
of $V=\bQ^n$. Then $(V^*)_{ng}^n$ is the set of $n$-tuples $(\l_0, \cdots, \l_{n-1}) \in (V^*)^n$ such that 
$\{\l_0, \cdots, \l_{n-1}\}$ is a $\bQ$-basis of $V^*$.
Then for $(\l_0, \cdots, \l_{n-1}) \in (V^*)^n_{ng}$,
there is a unique $\g \in G=\GL_n(\bQ)$ such that 
$$
\g \l_i = e_i^*, \text{ for } 0 \leq i \leq n-1.
$$

For $f \in \cS(V)$ and $(\l_0, \cdots, \l_{n-1}) \in (V^*)^n_{ng}$, we define a $\bQ$-linear map $\xi: (V^*)^n_{ng} \to Dist(V, \tilde K_M^n(V))$ by the following rule:
\begin{eqnarray} \label{kd}
\xi(\l_0, \cdots,\l_{n-1})(f) &:=&\left( \xi(e_1^*, \cdots, e_n^*)(\g\cdot f)\right) |_\g  \nonumber \\
&:=&\left(\sum_j b_j\{\eta (\chi_{a_j^{(1)} + d_j \bZ})(z_1),\eta (\chi_{a_j^{(2)} + d_j \bZ})(z_2), \cdots, \eta(\chi_{a_j^{(n)} + d_j\bZ})(z_n) \} \right) |_\g
%&:=&\left(\sum_j \{\eta (\chi_{a_j^{(1)} + d_j \bZ})(z_1)^{b_j},\eta (\chi_{a_j^{(2)} + d_j \bZ})(z_2)^{b_j}, \cdots, \eta(\chi_{a_j^{(n)} + d_j\bZ})(z_n)^{b_j} \} \right) |_\g
\end{eqnarray}
where we use the fact that $\g\cdot f$ can be written as 
$$
\g\cdot f=\sum_j b_j \cdot \chi_{\underline{a_j} + d_j \bZ^n}
$$ 
for some $b_j \in \bZ$ and $\underline{a_j}=(a_j^{(1),} \cdots, a_j^{(n)})$ and $d_j \in \bN$.

Here we view $\eta (\chi_{a_j^{(i)} + d_j \bZ})(z_i)=1-e\left(\frac{z_i-a_j^{(i)}}{d_j}\right)$ as
a function of $(z_1, \cdots, z_n) \in V_\bC=\C^n$. We need to show that the definition of $\xi$ does not depend on a presentation of a test function $f \in \cS(V)$.
If we write
$$
\underline{a} + d\bZ^n =
\bigcup_{\underline{b}} \underline{a} + d( \underline{b} +m\bZ^n),
\quad \underline{a} =(a_1, \cdots, a_n), \ \underline{b}=(b_1,\cdots, b_n),
$$
then
\begin{eqnarray*}
\chi_{\underline{a}+d\bZ^n} =\chi_{a_1 +  d\bZ} \otimes \cdots \otimes \chi_{a_n + d \bZ}&=&
(\sum_{b_1=0}^{m-1}\chi_{a_1 + d b_1  +dm \bZ}) \otimes \cdots \otimes (\sum_{b_n=0}^{m-1}\chi_{a_n + d b_n  +dm \bZ})\\
&=& \sum_{b_1, \cdots, b_n=0}^{m-1}
\chi_{\underline{a} + d \underline{b} + dm \bZ^n}.
%\bigotimes_{i=1}^n \chi_{a_i +db_i + dm \bZ}.
\end{eqnarray*}

\begin{eqnarray*}
\tilde\xi(e_1^*, \cdots, e_n^*)(\sum_{b_1, \cdots, b_n=0}^{m-1}
\chi_{\underline{a} + d \underline{b} + dm \bZ^n})
&=& \sum_{b_1, \cdots, b_n=0}^{m-1} \{\eta(\chi_{a_1 + d b_1  +dm \bZ)}(z_1),\cdots, \eta(\chi_{a_n + d b_n  +dm \bZ)}(z_n)\}\\
&=&\{\eta(\sum_{b_1=0}^{m-1}\chi_{a_1 + d b_1  +dm \bZ})(z_1), \cdots, \eta(\sum_{b_n=0}^{m-1}\chi_{a_1 + d b_n  +dm \bZ})(z_n)\}\\
&=&\{\eta(\chi_{a_1+d \bZ})(z_1), \cdots, 
\eta(\chi_{a_n + d \bZ})(z_n)\}.
\end{eqnarray*}

Note that $\l \in V^*$ induces a map $\l^*: \cR(\bQ) \to \cR(V)$ by the pullback.
%$\l(K)_n: \tilde K_n^M(\bQ)^* \to \tilde K_n^M(V)^*$ by the pullback.
In fact, if one wants to define $\xi \in \Hom((V^*)^n_{ng}, \mathrm{Dist}(V, \tilde K_n^M(V))$
such that $\xi|_\g = \xi, \forall \g \in \GL(V)$, then it is enough to define $\xi(e_1^*, \cdots, e_n^*) (f_1 \otimes \cdots \otimes f_n)$ for any factorisable $f=f_1\otimes \cdots \otimes f_n \in \cS(\bQ) \otimes \cdots \otimes \cS(\bQ) \cong \cS(V)$ by the $\GL(V)$-invariance and $\bZ$-linearity of $\xi$. 
Our definition \ref{kd} of $\xi(e_1^*, \cdots, e_n^*)(f)$ is the same as
$$
\xi(e_1^*, \cdots, e_n^*)(f_1 \otimes \cdots \otimes f_n) :=
\{(e_1^*)^*(\eta (f_1)), \cdots, (e_n^*)^*( \eta (f_n) )  \}=
\{\eta (f_1) (z_1), \cdots, \eta (f_n) (z_n)  \} \text{ in } \tilde K_n^M(V).
$$
This gives another proof that the definition \ref{kd} does not depend on the presentation of the test function.

\begin{thm}\label{mt}
The map $\xi$ induces a well-defined map 
$$
\xi: C_{n-1}(V) \to \operatorname{Dist}(V, \tilde K^M_n(V))
$$
such that 
%\textcolor{blue}{Suggestion: write $\partial$ instead of $\partial^*$ to simplify notation?}
$$
\partial^*(\xi) =0, \quad \xi |_\g = \xi \text{ for every } \g \in G=GL_n(\bQ).
$$
In other words, $\xi \in Z^{n-1}(V,\operatorname{Dist}(V, \tilde K_n^M(V))
)$.
\end{thm}

\begin{proof}

Now we show that $\xi$ factors through $C_{n-1}(V) = (V^*)^n_{ng}/(\bQ^+)^n$.
For $\underline b =(b_1, \cdots, b_n) \in (\bQ^+)^n$, let $dia(\underline b)$ be the diagonal matrix:
$$
dia(\underline b)=\begin{bmatrix}
    b_1 & 0 & \cdots & 0& 0 \\
    0 & b_2 &0 & \cdots & 0 \\
    0 & 0 & \cdots & \cdots &   0    \\
    %0 & 0 & \cdots & b_{n-1} &   0    \\
    0 & 0& \cdots & 0  & b_n
  \end{bmatrix}.
$$
For $\underline b=(b_1, \cdots, b_n) \in (\bQ^+)^n$, we compute
\begin{eqnarray*}
\xi(b_1 e_1^*, \cdots, b_n e_n^*)(\chi_{\underline{a}+d\bZ^n})&=&\xi(e_1^*,\cdots, e_n^*)(dia(\underline b)^{-1} \cdot \chi_{\underline{a}+d\bZ^n})|_{dia(\underline b)^{-1}} \\
&=& \xi(e_1^*,\cdots, e_n^*)(\chi_{b_1a_1 + b_1 d\bZ} \otimes \cdots \otimes \chi_{b_na_n + b_n d \bZ})|_{dia(\underline b)^{-1}} \\
&=& \{\eta(\chi_{b_1a_1 + b_1 d\bZ})(z_1), \cdots, \eta(\chi_{b_na_n + b_n d\bZ} )(z_n)  \}|_{dia(\underline b)^{-1}} \\
&=& \{\eta(\chi_{b_1a_1 + b_1 d\bZ})(b_1z_1), \cdots, \eta(\chi_{b_na_n + b_n d\bZ} )(b_nz_n)  \}|\\
&=& \{\eta(\chi_{a_1 +  d\bZ})(z_1), \cdots, \eta(\chi_{a_n +  d\bZ} )(z_n)  \}\\
&=& \xi(e_1^*, \cdots, e_n^*)(\chi_{\underline{a}+d\bZ^n}).
\end{eqnarray*}
A same computation shows that
$$
\xi(b_1 \l_0, \cdots, b_n \l_{n-1})(\chi_{\underline{a}+d\bZ^n})=\xi(\l_0, \cdots, \l_{n-1})(\chi_{\underline{a}+d\bZ^n}),
$$
for $(\l_0,\cdots, \l_{n-1}) \in (V^*)^n_{ng}$.
Therefore, $\xi$ factors through $C_{n-1}(V) = (V^*)^n_{ng}/(\bQ^+)^n$.

By the construction, the $GL(V)$-invariance follows directly:
$$
\xi|_\g =\xi, \quad \g \in GL(V).
$$

Finally, we prove the cocycle condition. For $[\l_0, \cdots, \l_n] \in C_n(V)$, the representative $(\l_0, \cdots, \l_n)$ satisfies that $(\l_0, \cdots, \hat{\l}_i, \cdots, \l_n) \in (V^*)^n_{ng}$ for any $i$.
The statement that $(\partial^*\xi) ([\l_0, \cdots, \l_n])(f)=0$ for every $[\l_0, \cdots, \l_n] \in C_n(V)$ and every $f \in \cS(V)$ is equivalent to the statement that 
$$
(\partial^* \xi) ([e_1^*+\cdots+e_n^*, e_1^*, \cdots, e_n^*]) (\chi_{\underline a +d \bZ^n}) =0, \quad \underline{a} \in \bQ^n, d \in \bN,
$$
because of the $GL(V)$-invariance and $\bZ$-linearity of $\xi$.
If we let 
$$
\a=\begin{bmatrix}
    1 & 0 & \cdots & 0& 0 \\
    1 & 1 &0 & \cdots & 0 \\
    %-1 & 0 & \cdots & \cdots &   0    \\
    \cdots & 0 & \cdots & 1  &   0    \\
    1 & 0& \cdots & 0  & 1
  \end{bmatrix},
$$
then $\a \cdot e_1^*= e_1^* +\cdots + e_n^*$ and $\a \cdot e_j^* = e_j^*$ for $j=2, \cdots, n$. We compute:
\begin{eqnarray*}
&&\xi([e_1^*+\cdots+e_n^*, e_2^*, \cdots, e_n^*])(\chi_{\underline{a}+d\bZ^n})\\
&=&\xi(\a\cdot [e_1^*,\cdots, e_n^*])(\chi_{\underline{a}+d\bZ^n}) \\
&=&\xi( [e_1^*,\cdots, e_n^*])(\chi_{\underline{a}\cdot \a^{}+d\bZ^n\cdot \a^{}})|_{\a^{-1}} \\
&=&\xi( [e_1^*,\cdots, e_n^*])(\chi_{(a_1+\cdots+a_n,a_2,\cdots, a_n)+d\bZ^n})|_{\a^{-1}} \\
&=& \{\eta(\chi_{a_1+ \cdots+a_n + d\bZ})(z_1), \eta(\chi_{a_2+d \bZ})(z_2),\cdots, \eta(\chi_{a_n + d\bZ})(z_n)   \}|_{\a^{-1}}  \\
&=& \{\eta(\chi_{a_1+ \cdots+a_n + d\bZ})(z_1+\cdots+z_n), \eta(\chi_{a_2+d \bZ})(z_2),\cdots, \eta(\chi_{a_n + d\bZ})(z_n)   \}. \\
\end{eqnarray*}
This computation implies that
\begin{eqnarray*}
&&(\partial^* \xi) ([e_1^*+\cdots+e_n^*, e_1^*, \cdots, e_n^*]) (\chi_{\underline a +d \bZ^n}) \\
&=& \{\eta(\chi_{a_1+ d\bZ})(z_1), \eta(\chi_{a_2+d \bZ})(z_2),\cdots, \eta(\chi_{a_n + d\bZ})(z_n)   \} \\
&+&\sum_{i=1}^n (-1)^i\{\eta(\chi_{a_1+ \cdots+a_n + d\bZ})(z_1+\cdots+z_n), \eta(\chi_{a_1+d \bZ})(z_1),\cdots, \widehat{\eta(\chi_{a_i + d\bZ})}(z_i) , \cdots,  \eta(\chi_{a_n + d\bZ})(z_n)   \} \\
&=& \{1-e(\frac{z_1-a_1}{d}), 1-e(\frac{z_2-a_2}{d}),\cdots, 1-e(\frac{z_n-a_n}{d})   \} \\
&+&\sum_{i=1}^n (-1)^i\{1-e\left(\frac{z_1+\cdots + z_n - (a_1+ \cdots+a_n )}{d}\right), 1-e(\frac{z_1-a_1}{d}),\cdots, \widehat{1-e(\frac{z_i-a_i}{d})} , \cdots,  1-e(\frac{z_n-a_n}{d})  \} \\
&\equiv& 
\sum_{i=0}^n (-1)^i \{u_0, \cdots, \hat{u}_i, \cdots, u_n \} \quad (\mod J),
\end{eqnarray*}
where 
\begin{eqnarray*}
&& u_0:=1-e\left(\frac{z_1+\cdots + z_n - (a_1+ \cdots+a_n )}{d}\right), \quad u_1:=1-e(\frac{z_1-a_1}{d}),\\
&& u_i:= \left(1-e(\frac{z_i-a_i}{d})\right) \prod_{m=1}^{i-1} e(\frac{z_m-a_m}{d}).
\end{eqnarray*}
Then $u_0 = u_1 + \cdots + u_n$ and $u_1 + \cdots + u_k \in \cR(V)^\times$ for each $k \geq 1$. Then Theorem \ref{drl} (Dedekind reciprocity law) implies that (see Definition \ref{targetk} for $\tilde J$)
$$
\sum_{i=0}^n (-1)^i \{u_0, \cdots, \hat{u}_i, \cdots, u_n \}  \in \tilde J.
$$
This proves the cocycle property
$$
\partial^*(\xi) \equiv 0  \text{ in }  \operatorname{Dist}(V,\tilde K_n^M(V)).
$$
\end{proof}

\begin{defn}\label{Eisn}
For each $n$, we define $\xi_n^{St}$ to be the element 
$$
\xi \in Z^{n-1}(V, \operatorname{Dist}(V,\tilde K_n^M(V)))\subseteq\Hom_{\GL(V)}(C_{n-1}(V),
\operatorname{Dist}(V,\tilde K_n^M(V)))
$$ 
in Theorem \ref{mt}.
\end{defn}

When $n=1$, then $ \xi_1^{St}$ belongs to $\Hom_{\bQ^\times}( C_0(\bQ), 
\mathrm{Dist}(\bQ,  \tilde K_1^M(\bQ)))$. Note that 
$ C_0(\bQ)$ is a free abelian group of rank 2 generated $[ e_1^* ]$ and 
$[-e_1^*]$
and $\tilde K_1^M(\bQ) = K_1^M(\bQ)/\tilde J_1(\bQ)$. We have
$$
 \xi_1^{St} ([ e_1^* ]) (\chi_{a + d\bZ}) = \{1- e \left(\frac{z-a}{d} \right)\}, \quad a, d \in \bQ^+, a <d.
$$
On the other hand,
$$
 \xi_1^{St} ([-e_1^*]) (\chi_{a + d \bZ}) =\{\left(  \xi_1^{St} (e_1^*) (\chi_{-a - d \bZ}) \right)\}|_{-1}=\{\left(  \xi_1^{St} (e_1^*) (\chi_{(d-a) + d \bZ}) \right)\}|_{-1}
=\{1-e\left(\frac{-z+a}{d} \right)\}.
$$
Because
$$
-e \left(\frac{z-a}{d} \right) \cdot \left(1-e\left(\frac{-z+a}{d} \right)\right) = 1- e \left(\frac{z-a}{d} \right),
$$
we have that 
$$
 \xi_1^{St} (e_1^*) (\chi_{a + d\bZ}) -  \xi_1^{St} (-e_1^*) (\chi_{a + d \bZ}) 
\in \tilde J_1(\bQ)
$$ 
and thus $ \xi_1^{St} ([e_1^*]) = \xi_1^{St} ([-e_1^*]) $.

\begin{defn}
Denote the image of the Stevens cocycle $\xi_n^{St}$ (given in Definition~\ref{Eisn}) under the map in Corollary~\ref{GpCocycle} by $\Phi_n^{St} := \Phi_{e_1^*}(\xi_n^{St})$, which is a homogeneous group cocycle. We also refer to this as the Stevens cocycle:
\begin{eqnarray*}
& \Phi_n^{St}  : (\GL_n \Q)^n \rightarrow \on{Dist} ( \Q^n, \tilde{K}_n^M(V) ) \\
& \Phi_n^{St}(\gamma_1, \cdots \gamma_n)(f) =\xi_n^{St}([\gamma_1 \cdot e_1^*, \cdots \gamma_n \cdot e_1^* ] )(f).
\end{eqnarray*}
\end{defn}

\section{The Shintani group cocycle and the comparison}
\label{sec3}
%\subsection{Modular symbol for $\GL_2(\bQ)$.}

\subsection{The Shintani cocycle}\label{sec3.1}

The Shintani cocycle, defined per each natural number $n$, is a homogeneous group $n$-cocycle:
    \begin{align*}
        [\Phi^{Sh}_n] \in H^{n-1}\left(\on{GL}_n(\RR), \operatorname{Dist} \left(\bA_f^n \backslash \{0\}, \bC((T_1, \cdots T_n))  \right)\right),
    \end{align*}
    where $\bA_f$ is the ring of finite adeles of $\bQ$
 and $\bC(( T_1, \cdots, T_n))$ is the ring of Laurent power series with $n$-variables $T_1, \cdots, T_n$.
We briefly recall its definition following \cite{H}, which we refer to the reader for more details.

Let $\mathcal L_{\Q^n}$ be the abelian group generated by characteristic functions of open rational cones, modulo constant functions. We recall \emph{the Solomon-Hu pairing} from \cite{H}:
%Also $\langle \cdot, \cdot \rangle_{SH}$ is the Solomon-Hu pairing:
\begin{align}\label{SH}
\begin{split}
& \langle \cdot , \cdot \rangle_{SH} : \cL_{\QQ^n} \times \cS(\bA_f^n \backslash \{0\}) \rightarrow \bC((T_1, \cdots T_n )) \\
& \langle c, f \rangle_{SH} := \frac1{1-e^{v_1 T_1}} \cdots \frac1{1-e^{v_r \cdot T_r}} \sum_{w \in \cP_{c,f} \cap \Q^n} f(w) e^{w \cdot T}
\end{split}
\end{align}
where $c=\chi_{\bR^+ v_1 + \cdots + \bR^+ v_r}$ and vectors $\{ v_1, \cdots v_r \} \subset \Q^n$ are assumed to belong to the lattice of periods of test function $f$ (if $L_f$ is the lattice of periods of $f$, by linear dependence $\forall v \in \QQ^n, \exists n: n \cdot v \in L_f$) and $\cP_{c,f} := (0,1] v_1 + \cdots + (0,1] v_r$ is the fundamental parallelogram for $c$. Note that the pairing formally expresses the quantity $\sum_{w \in \Q^n} c(w) f(w) e^{w \cdot z}$, but this quantity is not well-defined as a formal power series and thus we `multiply' it by $\prod_j (1-e^{v_j z_j})^{-1}$ to obtain a finite sum $\sum_{w \in P_{c, f}} f(w) e^{w \cdot z}$.
%\begin{align*}
%& \langle \cdot , \cdot \rangle_{SH} : \cL_{\Q^n} \times \cS(\bA_f^n \backslash \{0\}) \rightarrow \bC((T_1, \cdots T_n )) \\
%& \langle c, f \rangle_{SH} := \frac1{1-e^{v_1 T_1}} \cdots \frac1{1-e^{v_r \cdot T_r}} \sum_{w \in P_{\mc}_{c,f} \cap \Q^n} f(w) e^{w \cdot T}
%\end{align*}
%where $c=\chi_{\RR^+ v_1 + \cdots + \RR^+ v_r}$ and vectors $\{ v_1, \cdots v_r \} \subset \Q^n$ are assumed to belong to the lattice of periods of test function $f$ (if $L_f$ is the lattice of periods of $f$, by linear dependence $\forall v \in \Q^n, \exists n: n \cdot v \in L_f$) and $P_{\mc}_{c,f} := (0,1] v_1 + \cdots + (0,1] v_r$ is the fundamental parallelogram for $c$. Note that the pairing formally expresses the quantity $\sum_{w \in \Q^n} c(w) f(w) e^{w \cdot z}$, but this quantity is not well-defined as a formal power series and thus we `multiply' it by $\prod_j (1-e^{v_j z_j})^{-1}$ to obtain a finite sum $\sum_{w \in P_{\mc}_{c, f}} f(w) e^{w \cdot z}$.

The Shintani cocycle is defined by supplying this pairing with an appropriate polyhedral cone function $\sigma_n^{Sh}$:
$$\Phi_n^{Sh}:= \langle \sigma_n^{Sh}, \cdot \rangle_{SH}.$$
The polyhedral cone function $\sigma_n^{Sh}$ (see \ref{pcf} below) is defined as a cone function in $\FF^n$ for some ordered field $\FF$ followed by the restriction to $\R^n$ by the embedding $\R^n \rightarrow \FF^n$.
Define $\FF = \RR((\epsilon_1)) \cdots ((\epsilon_n))$ and give its monomials lexicographic ordering: if $\epsilon^r = \epsilon_1^{r_1} \cdots \epsilon_n^{r_n}$ and $\epsilon^s = \epsilon_1^{s_1} \cdots \epsilon_n^{s_n}$ are two monomials, then $\epsilon^r \succ \epsilon^s$ iff for some $1 \le i \le n$, $r_n = s_n, r_{n-1} = s_{n-1}, \cdots r_{i+1} = s_{i+1}, r_i < s_i$. When $\epsilon^r \succ \epsilon^s$, we say that $\epsilon^r$ succedes $\epsilon^s$. In a linear combination of such monomials, the most succeding term is called the leading term.

Order elements of $\FF$ by defining $r>0$ ($r \in \FF$) iff the coefficient of its leading term is positive, and defining $r>s \iff r-s>0$. This is a total order that is well-behaved under addition and positive multiplication.
Here we are using the opposite convention from that of \cite{H}; preceding relation there is changed to succeeding relation. This is done so that $\succ$ represents magnitude of monomials when we think of $\epsilon_1, \epsilon_2, \epsilon_3 \cdots$ as infinitesimals that get progressively smaller. For example,
$$\cdots \succ \epsilon_2^{-2} \succ \epsilon_2^{-1} \succ \cdots \succ \epsilon_1^{-2} \succ \epsilon_1^{-1} \succ \cdots \succ 1 \succ \cdots \succ \epsilon_1 \succ \epsilon_1^2 \succ \cdots \succ \epsilon_2 \succ \epsilon_2^2 \succ \cdots .$$
As an illustration, $\epsilon_1^{1000} \succ \epsilon_2$ since $\epsilon_2$ is another magnitude smaller than $\epsilon_1$, no matter `how small $\epsilon_1$ tries to get'. The sign of an element of $\FF$ is determined by the coefficient of its leading term, which is intuitively the monomial of the largest magnitude. For example, $\epsilon_1^{1000} - 1000 \epsilon_2 > 0$ since $\epsilon_2$ is another magnitude smaller than $\epsilon_1^{1000}$ regardless the coefficient $-1000$. 

Define a polyhedral cone function $\sigma_n^{Sh}$ (see section 4, \cite{H} for details):
\begin{align}\label{pcf}
\begin{split}
    & [\sigma_n^{Sh}] \in H^{n-1}(\on{GL}_n(\RR), \cL_{\Q^n}) \\
    & \sigma_n^{Sh} (\alpha_1, \cdots \alpha_n)(w) := c(\alpha_1 b(\epsilon_1), \cdots \alpha_n b(\epsilon_n))(w), \quad w \in \bR^n\setminus \{0\}.
\end{split}    
\end{align}
Here $b$ is defined by
$$b(\epsilon) := \begin{bmatrix} 1 & \epsilon & \cdots & \epsilon^{n-1} \end{bmatrix}^\text{T}$$
and $c$ is a cone function defined per each basis $(v_1, \cdots v_n) \subset \FF^n$ given by
\begin{align}\label{sign}
\begin{split}
    & c(v_1, \cdots v_n) : \FF^n \rightarrow \ZZ \\
    & c: (v_1, \cdots v_n): w \mapsto \begin{cases} \on{sign}\on{det}(v_1, \cdots v_n) & \text{if $w = \sum_j \lambda_j v_j$ with $\forall j, \lambda_j >0$ } \\ 0 & \text{ otherwise} \end{cases}
\end{split}  
\end{align}
(sign of determinant is given precisely by the ordering of $\FF$). The proof that $c$ is indeed a polyhedral cone function and $\sigma_n^{Sh}$ is a cocycle with values in $\cL_{\bQ^n}$ was given in \cite{H}. The delicate sign conventions in \ref{sign} is the key to make $\sigma_n^{Sh}$ a group cocycle.

%We have the following key calculation:

\subsection{The Naive Shintani cocycle}\label{sec3.2}

The naive Shintani cocycle is a `naive' version of the Shintani cocycle, which is not a cocycle at first but becomes a cocycle modulo certain elements. The fact that the naive Shintani cocycle is a cocycle will be a corollary of our comparison result between the naive Shintani cocycle and the Stevens cocycle.

The following result on the usual Shintani cocycle motivates us to define the Naive Shintani function (Definition \ref{nsh}).
\begin{prop} \label{rhos}
    Let $\rho \in \GL_n(\RR)$ be the shift permutation defined by $\rho_{1,n} = 1, \rho_{i+1, i} = 1$ for $i=1, \cdots n-1$ and all other entries zero. Then
    \begin{align*}
        \sigma_n^{Sh} (1, \rho, \cdots \rho^{n-1}) = \chi_{(\RR^+)^n}.
    \end{align*}
\end{prop}
\begin{proof}
Denote $\sigma_0 = \sigma_n^{Sh} (1, \cdots, \rho^{n-1})$. Let $M = [1 \cdot b(\epsilon_1), \cdots \rho^{n-1} \cdot b(\epsilon_n)]$. Since $\rho^k$ is permutation by shifting $k$ entries, $M_{i+j,j} = \epsilon_j^i$ where the index $i+j$ is modulo $n$. In $\FF = \RR((\epsilon_1)) \cdots ((\epsilon_n))$, denote the leading monomial of an element $a \in \FF$ by $\on{Lead}(a)$. Nonzero value of $\tilde \sigma_0$ is $\sign \det M = 1$ and it remains to find where $\tilde \sigma_0$ is nonzero.

We first simplify the required computation. $\tilde \sigma_0(w)$ is nonzero iff the $v$ such that $Mv=w$ satisfies $v>0$ ($M$ is nonsingular, as shown in \cite{H}). This is again equivalent to $M^{-1}w>0$ (we write $w>0$ if all entries are positive). We have $\sign (M^{-1} w) = \sign (\det (M)) \on{adj}(M) w = (\sign \det M) \sign \on{adj}(M)w$. Also note that we only need to care about leading term of each entry of $\on{adj}(M)$; the equation $\on{adj}(M)w>0$ consists of $n$ equations induced by each row, and as each row equation consists of $n \cdot (n-1)!$ distinct monomials, it suffices to consider the leading terms. 

We now directly compute $\xi_{ij}:= \on{Lead}(\on{adj}(M)_{ij})$:
\begin{align*}
    \xi_{1,1} &= 1 \\
    \xi_{1,i} &= -\epsilon_i^{n-i+1} \text{ if $i \ge 2$}
\end{align*}
and when $k \ge 2$,
\begin{align*}
    \xi_{k,i} &= -\epsilon_i^{k-i} \text{ if $1 \le i \le k-1$} \\
    \xi_{k,k} &= 1 \\
    \xi_{k,i} &= \epsilon_1^{k-1} \epsilon_i^{n-i+1} \text{ if $k+1 \le i \le n$}.
\end{align*}
Thus
\begin{align*}
    & \xi_{1,1} \prec \cdots \prec \xi_{1,n} \\
    & \xi_{k,k} \prec \xi_{k,1} \prec \cdots \prec \xi_{k,k-1} \prec \xi_{k,k+1} \prec \cdots \prec \xi_{k,n}.
\end{align*}

Denoting $S_\pm (i_1, \cdots i_k) = \{w_{i_1} = \cdots = w_{i_{k-1}}=0, \pm w_{i_k} >0\}$, the row equations give:
\begin{align*}
    \text{$1$st row: } R_1 &= S_+(1) \cup  \bigcup_{i=2}^n S_-(1, \cdots i) \\
    \text{$1<k$th row: } R_k &= S_+(k) \cup \bigcup_{i=1}^{k-1} S_-(k, 1, \cdots i) \cup \bigcup_{i=k+1}^n S_+(1, \cdots i).
\end{align*}
Our claim is that 
\begin{align*}
    R_1 \cap \cdots \cap R_n &= S_+(1) \cap \cdots \cap S_+(n) = (\RR^+)^n.
\end{align*}
We can prove this inductively by showing $R_1 \cap \cdots \cap R_k = S_+(1) \cap \cdots \cap S_+(k)$ for $k\ge 2$. The base case $k=2$ is not hard to check. To prove the statement for $k+1$ from $k$, note that at least one of the first $k$ coordinates of $R_{k+1} \backslash S_+(k+1)$ is non-positive, which has no intersection with $S_+(1) \cap \cdots \cap S_+(k)$.
\end{proof}

\begin{defn}\label{nsh}
Define a function $\Phi_n^{NSh},$ which we call the Naive Shintani function:
\begin{eqnarray*}
&\Phi_n^{NSh}  : (\GL_n \Q)^n \rightarrow \on{Dist} ( \Q^n, \bC (( T_1, \cdots T_n )) ) \\
 &\Phi_n^{NSh}(\gamma_1, \cdots \gamma_n)(f) = \langle \chi_{\RR^+ e_1 \gamma_1^T + \cdots + \RR^+ e_1 \gamma_n^T }, f \rangle_{SH},
\end{eqnarray*}
so that 
$$
  \Phi_n^{Sh} (1, \rho, \cdots \rho^{n-1}) =   \Phi_n^{NSh} (1, \rho, \cdots \rho^{n-1}).
$$
%Here, $\bC (( T_1, \cdots T_n))$ is defined as the fraction field of $\bC [[T_1, \cdots T_n ]]$. 
\end{defn}

We emphasize that the Shintani cocycle develops machinery (such as \ref{sign}) to replace the cone function $\chi_{\bR^+ e_1 \gamma_1^T + \cdots + \bR^+ e_1 \gamma_n^T}$ into a cone function that includes faces of the cone. Such efforts make sure that the map thus defined is indeed a homogeneous group cocycle. Our naive Shintani cocycle is not a cocycle by itself.

%\section{The comparison}

\subsection{Comparison between the Naive Shintani cocycle and the Stevens cocycle}\label{sec3.3}

Here we give a comparison result between the naive Shintani cocycle and the Stevens cocycle. 
In order to compare the naive Shintani cocycle to the Stevens cocycle, we recall the dlog map. First define:
\begin{eqnarray*}
\dlog : & \cR (V)^\times  & \rightarrow \mathbb C((z_1 , \cdots z_n)) \mathrm{d} z_1 \oplus \cdots \oplus \mathbb C((z_1 , \cdots z_n)) \mathrm{d} z_n \\
 & f & \mapsto \frac1f (\frac{\partial f}{\partial z_1} \mathrm{d}z_1 + \cdots + \frac{\partial f}{\partial z_n} \mathrm{d}z_n).
\end{eqnarray*}

The dlog map satisfies $\dlog (fg) = \dlog (f) + \dlog (g)$ and thanks to this we can also define the wedged dlog map (with the same notation):
\begin{eqnarray*}
{\dlog}: & \cR(V)^\times \otimes \cdots \otimes \cR(V)^\times & \rightarrow \mathbb C((z_1, \cdots z_n)) \mathrm{d} z_1 \wedge \cdots \wedge \mathrm{d} z_n \\
 & f_1 \otimes \cdots \otimes f_n & \mapsto \dlog f_1 \wedge \cdots \wedge \dlog f_n.
\end{eqnarray*}

Furthermore it can be easily checked that the wedged dlog map 
factors through $K_n^M(V)$:
\begin{eqnarray}\label{dlo}
\begin{split}
{\dlog}: & K_n^M(V) & \rightarrow & \ \mathbb C((z_1, \cdots z_n)) \mathrm{d} z_1 \wedge \cdots \wedge \mathrm{d} z_n \\
& \{f_1, \cdots, f_n\} & \mapsto &\  \dlog f_1 \wedge \cdots \wedge \dlog f_n.
\end{split}
\end{eqnarray}
and becomes an additive homomorphism:
\begin{eqnarray*}
%\begin{split}
{\dlog}(  \{x_1, x_2, \cdots, x_n\} +\{x_1', x_2, \cdots, x_n  \}) 
&=& \wedge^n (\{x_1 \cdot x_1', x_2, \cdots, x_n\}) \\
&=&\left(\dlog(x_1) + \dlog(x_1') \right) \wedge \dlog(x_2)\wedge \cdots \wedge \dlog(x_n) \\
&=& {\dlog}(  \{x_1, x_2, \cdots, x_n\}) +{\dlog}(\{x_1', x_2, \cdots, x_n  \}). 
%\end{split}
\end{eqnarray*}

Let us define 
\begin{align}\label{jnp}
J_n' := {\dlog}\left( \tilde J_n(V) \right)
\end{align}
where $\tilde J_n(V)$ was given in Definition \ref{targetk}.
%The map doesn't factor through $\tilde K_n^M(V) = K_n^M(V) / \tilde J_n(V)$, but we may define {\color{blue} $J_n' = \bigoplus_{j=1}^n (T_j) \mathrm{d}T_1 \wedge \cdots \wedge \frac{\mathrm{d} T_j}{T_j} \wedge \cdots \wedge \mathrm{d} T_n$ such that ${\dlog}(\tilde J_n(V)) \subseteq J_n'$} and 
Then we get an additive homomorphism (with the same notation ${\dlog}$)
\begin{eqnarray*}
{\dlog}: \tilde K_n^M(V)=K_n^M(V)/\tilde J_n(V) & \rightarrow \mathbb C((z_1, \cdots z_n)) \mathrm{d} z_1 \wedge \cdots \wedge \mathrm{d} z_n / J_n'. 
 %& \{f_1, \cdots f_n\} & \mapsto [\dlog f_1 \wedge \cdots \wedge \dlog f_n]
\end{eqnarray*}
We now give our comparison result:
\begin{thm}\label{EisNSh}
For given $(\gamma_1, \cdots, \gamma_n) \in \GL_n (\Q)^n$ and $f \in \cS (\Q^n)$, 
\begin{eqnarray*}
{\dlog}( \Phi_n^{St}(\gamma_1, \cdots, \gamma_n)(f)) \equiv (-1)^n \cdot (\Phi_n^{NSh} (\gamma_1, \cdots, \gamma_n)(\hat f)) \mathrm{d} T_1 \wedge \cdots \wedge \mathrm{d} T_n  \quad (\operatorname{mod} J_n'),
\end{eqnarray*}
 under the identification $T_j = 2\pi i z_j$, $j=1, \cdots, n$.
\end{thm}

%Before proving the result in full, we first introduce some preliminary computations. Here, 

We need two lemmas for its proof. 
    Recall that $\rho \in \GL_n(\RR)$ is the shift permutation matrix defined by $\rho_{1,n} = 1, \rho_{i+1, i} = 1$ for $i=1, \cdots n-1$ and all other entries zero.
Since $\Phi_n^{St}(\rho^0, \rho^1, \cdots, \rho^{n-1})$ and $\Phi_n^{NSh}(\rho^0, \rho^1, \cdots, \rho^{n-1})$ will play a central role, we give some shorthand notations:
\begin{eqnarray*}
& &\mu_n^{St} = \Phi_n^{St}(\rho^0, \rho^1, \cdots \rho^{n-1}), \quad \mu_n^{NSh} = \Phi_n^{NSh}(\rho^0, \rho^1, \cdots \rho^{n-1}),\\
& & \omega_T = \mathrm{d}T_1 \wedge \cdots \wedge \mathrm{d} T_n, \quad \omega_z = \mathrm{d} z_1 \wedge \cdots \wedge \mathrm{d} z_n.
\end{eqnarray*}

\begin{lem}\label{lemo}
    Under the identification $T_j = 2\pi i z_j$, for any test function $f \in \cS(\Q^n)$,
    \begin{align*}
        {\dlog}(\mu_n^{St}(f)) \equiv (-1)^n \mu_n^{NSh}(\hat f) \omega_T.
    \end{align*}
\end{lem}
\begin{proof}
Let us first prove the 1-dimensional case:
    \begin{align*}
        \dlog \xi_1^{St} ([e_1^* ])(f) \equiv -\Phi_1^{NSh}(1)(\hat{f}) \mathrm{d} T.
    \end{align*}
In general, $\chi_{a\ZZ} = \sum_{j=0}^{b-1} \chi_{ja + ab \ZZ}$ holds. For $a = \frac{a_1}{a_2}, d=\frac{d_1}{d_2}$ where $a_j, d_j \in \ZZ$,
    \begin{align*}
        \hat{\chi}_{a+d\ZZ}(y) =& \frac1d e^{-2\pi i a y} \chi_{\frac1d \ZZ} (y)
        = \frac{1}{d} \sum_{j=0}^{d_1 a_2 - 1} e^{-2\pi i ay} \chi_{\frac{j}{d}+a_2 d_2 \ZZ}(y)
        = \frac 1d \sum_{j=0}^{d_1 a_2 - 1} e^{-2\pi i \frac ad j} \chi_{\frac jd + a_2 d_2 \ZZ}(y).
    \end{align*}
    Denote $(a)' := \chi_\ZZ(a) + \{ a \} = \chi_\ZZ(a) + a - \lfloor a \rfloor$. For a test function of the form $\chi_{a+d\ZZ}$, we obtain the following comparison:
\begin{align*}
\xi_1^{St} (\langle e_1^*\rangle ) (\chi_{a+d\ZZ}) =& \{ 1 - e^{2\pi i \frac{z-a}{d} } \} \implies \dlog \xi_1^{St} (\langle e_1^* \rangle )(\chi_{a+d\ZZ}) = \frac{2\pi i }{d} \frac{-e^{2\pi i \frac{z-a}{d}}}{1-e^{2\pi i \frac{z-a}{d}}} \mathrm{d} z = -\frac{1}{d} \frac{e^{2\pi i \frac{z-a}{d}}}{1-e^{2\pi i \frac{z-a}{d}}} \mathrm{d} T
\end{align*}
and
    \begin{align*}
        \Phi_1^{NSh}(1)(\hat{\chi}_{a+d\ZZ}) =& \langle \chi_{\RR^+}, \hat{\chi}_{a+d\ZZ} \rangle \\
        =& \frac1d \sum_{j=0}^{d_1 a_2 - 1} e^{-2\pi i \frac ad j} \langle \chi_{\RR^+},  \chi_{\frac jd + a_2 d_2 \ZZ} \rangle \\
        =& \frac1d \sum_{j=0}^{d_1 a_2 - 1 } e^{-2\pi i \frac ad j}\frac{e^{(\frac j{d a_2 d_2})'\cdot a_2 d_2 T}}{1-e^{a_2 d_2 T}} \\
        =& \frac1{d(1-e^{a_2 d_2 T})} \sum_{j=0}^{d_1 a_2 - 1} e^{(\frac j{d_1 a_2})'\cdot a_2 d_2 T - 2\pi i \frac ad j} \\
        =& \frac1{d(1-e^{a_2 d_2 T})} \left( e^{a_2 d_2 T} + \sum_{j=1}^{d_1 a_2 - 1} e^{ \frac jd \cdot T - 2\pi i \frac ad j} \right) \\
        =& \frac1{d(1-e^{a_2 d_2 T})} \left( e^{a_2 d_2 T} + \sum_{j=1}^{d_1 a_2 - 1} \omega^j \right) \text{ (where $\omega = e^{\frac{T-2\pi i a}{d}}$)} \\
        =& \frac1{d(1-\omega^{d_1 a_2})} \left( \omega^{d_1 a_2} + \frac{\omega^{d_1 a_2} - \omega}{\omega - 1} \right) \\
        =& \frac1{d(1-\omega^{d_1 a_2})} \frac{\omega^{d_1 a_2 + 1} - \omega}{\omega - 1} \\
        =& \frac1d \frac \omega{1-\omega} \\
        =& \frac1d \frac{e^{\frac{T-2\pi i a}{d}}}{1-e^{\frac{T-2\pi i a}{d}}}.
    \end{align*}
    So for any test function $f$, by decomposing it into functions of the form $\chi_{a+d\ZZ}$, we see that
    \begin{align*}
        \dlog \xi(\langle e_1^* \rangle)(f) = \frac1d \frac{e^{2\pi i \frac{z-a}d }}{1-e^{2\pi i \frac{z-a}d }} (2\pi i \mathrm{d} z) = - \frac1d \frac{e^{\frac{T-2\pi i a}d }}{e^{\frac{T-2\pi i a}d }-1} \mathrm{d} T = -\Phi_{Sh}(1)(\hat{f}) \mathrm{d} T.
    \end{align*}
This proves the 1-dimensional case. More generally for $f = \chi_{a_1 + d_1 \ZZ} \otimes \cdots \otimes \chi_{a_n + d_n \ZZ}$,
    \begin{align*}
        \xi_n^{St} ([ e_1^* , \cdots, e_n^* ])(f) =& \{1-e^{2\pi i \frac{z_1-a_1}{d_1}}, \cdots, 1-e^{2\pi i \frac{z_n-a_n}{d_n}}\} \\
        \implies {\dlog}\xi_n^{St} ([ e_1^* , \cdots, e_n^* ])(f) =& \prod_{j=1}^n \frac1{d_j} \frac{-e^{2\pi i \frac{z_j-a_j}{d_j}}}{1-e^{2\pi i \frac{z_j-a_j}{d_j}}} (2\pi i \mathrm{d} z_1) \wedge \cdots \wedge (2\pi i \mathrm{d} z_n) \\
        =& (-1)^n \prod_{j=1}^n \frac1{d_j} \frac{e^{\frac{T_j-2\pi i a_j}{d_j}}}{1-e^{\frac{T_j-2\pi i a_j}{d_j}}} \omega_T \\
        =& (-1)^n \prod_{j=1}^n \Phi_1^{NSh}(1)(\hat{\chi}_{a_j + d_j \ZZ}) \omega_T \\
        =& (-1)^n \Phi_n^{NSh}(1, \rho, \cdots, \rho^{n-1})(\hat{\chi}_{a_1 + d_1 \ZZ} \otimes \cdots \otimes \hat{\chi}_{a_n + d_n \ZZ}) \omega_T \\
        =& (-1)^n \Phi_n^{NSh}(1, \rho, \cdots, \rho^{n-1})(\hat f) \omega_T
    \end{align*}
where we used $\cF(\chi_{a_1 + d_1 \ZZ} \otimes \cdots \otimes \chi_{a_n + d_n \ZZ}) = \hat{\chi}_{a_1 + d_1 \ZZ} \otimes \cdots \otimes \hat{\chi}_{a_n + d_n \ZZ}$. By writing arbitrary test function as a sum of functions of the form $\chi_{a_1 + d_1 \ZZ} \otimes \cdots \otimes \chi_{a_n + d_n \ZZ}$, we get the desired result.
\end{proof}

\begin{lem} \label{lemt}
For any $\g \in \GL_n(\bQ)$, we have
    $${\dlog}\Phi_n^{St}(\gamma \rho^0, \cdots, \gamma \rho^{n-1})(f) = \gamma \cdot {\dlog}\Phi_n^{St} (\rho^0, \cdots, \rho^{n-1})(\gamma^{-1} \cdot f).$$
\end{lem}
\begin{proof}
    Let $\gamma^{-1} \cdot f = \chi_{a + d\ZZ^n}$. Denote columns of $\gamma$ by $\gamma^{(1)}, \cdots, \gamma^{(n)}$. Then
    \begin{align*}
 &       {\dlog}\Phi_n^{St}(\gamma \rho^0, \cdots, \gamma \rho^{n-1})(f)\\ =& {\dlog}\left[ (\mu_n^{St}|_{\gamma^{-1}} )(f) \right] \\
        =& {\dlog}\left[ \mu_n^{St}(\gamma^{-1} \cdot f)|_{\gamma^{-1}} \right] \\
        =& {\dlog}\left[ \left\{ 1-e^{2\pi i \frac{z_1 - a_1}{d}} , \cdots, 1-e^{2\pi i \frac{z_n - a_n}{d}} \right\}|_{\gamma^{-1}} \right] \\
        =& {\dlog}\left[ \left\{ 1-e^{2 \pi i \frac{z\cdot \gamma^{(1)} - a_1}{d}}, \cdots, 1-e^{2\pi i \frac{z\cdot \gamma^{(n)} - a_n}{d}} \right\} \right] \\
        =& \frac{-e^{2\pi i \frac{z\cdot \gamma^{(1)} - a_1}d}}{1-e^{2\pi i \frac{z \cdot \gamma^{(1)} - a_1}d}} \cdots \frac{-e^{2\pi i \frac{z\cdot \gamma^{(n)} - a_n}d}}{1-e^{2\pi i \frac{z \cdot \gamma^{(n)} - a_n}d}} \frac{(2\pi i )\mathrm{d} z \cdot \gamma^{(1)}}{d} \wedge \cdots \wedge \frac{(2\pi i ) \mathrm{d} z \cdot \gamma^{(n)}}{d} \\
        =& \frac{-e^{2\pi i \frac{z\cdot \gamma^{(1)} - a_1}d}}{1-e^{2\pi i \frac{z \cdot \gamma^{(1)} - a_1}d}} \cdots \frac{-e^{2\pi i \frac{z\cdot \gamma^{(n)} - a_n}d}}{1-e^{2\pi i \frac{z \cdot \gamma^{(n)} - a_n}d}} \frac1{d^n} (\det \gamma) (2\pi i )\mathrm{d} z_1 \wedge \cdots \wedge (2\pi i )\mathrm{d} z_n
    \end{align*}
    and
    \begin{align*}
   &     \gamma \cdot {\dlog}\Phi_n^{St} (\rho^0, \cdots, \rho^{n-1})(\gamma^{-1} \cdot f) \\
        =& \gamma \cdot {\dlog}[ \mu_n^{St} (\gamma^{-1} \cdot f) ] \\
        =& \gamma \cdot {\dlog}\left( \left\{ 1 - e^{2\pi i \frac{z_1-a_1}d}, \cdots, 1 - e^{2\pi i \frac{z_n-a_n}d} \right \} \right) \\
        =& \gamma \cdot \left( \frac{-e^{2\pi i \frac{z_1 - a_1}d }}{1-e^{2\pi i \frac{z_1 - a_1}d}} \cdots \frac{-e^{2\pi i \frac{z_n - a_n}d} }{1-e^{2\pi i \frac{z_n - a_n}d}} \frac1{d^n} (2\pi i ) \mathrm{d} z_1 \wedge \cdots \wedge (2\pi i ) \mathrm{d} z_n \right) \\
        =& \frac{-e^{2\pi i \frac{z\cdot \gamma^{(1)} - a_1}d}}{1-e^{2\pi i \frac{z \cdot \gamma^{(1)} - a_1}d}} \cdots \frac{-e^{2\pi i \frac{z\cdot \gamma^{(n)} - a_n}d}}{1-e^{2\pi i \frac{z \cdot \gamma^{(n)} - a_n}d}} \frac1{d^n} (\det \gamma) (2\pi i )\mathrm{d} z_1 \wedge \cdots \wedge (2\pi i ) \mathrm{d} z_n.
    \end{align*}
Thus the two are equal.
\end{proof}

\begin{proof} (of Theorem~\ref{EisNSh})
Let $\gamma$ be the matrix defined by having its $j$th column equal to the first column of $\gamma_j$.  
%By assumption $\gamma$ is orthogonal. Then 
Because $\forall j, \gamma_j \cdot e_1^* = \gamma \rho^j \cdot e_1^*$ and $\forall j, e_1\cdot \gamma_j^T = e_1 \cdot (\g\rho^{j})^T$, we have
\begin{eqnarray}\label{idy}
\begin{split}
& \Phi_n^{St}(\gamma_1, \cdots, \gamma_n)(f) = \Phi_n^{St}(\gamma \rho^0 , \cdots, \gamma \rho^{n-1})(f) \\
& \Phi_n^{NSh}(\gamma_1, \cdots, \gamma_n)(f) = \Phi_n^{NSh}(\gamma \rho^0 , \cdots, \gamma \rho^{n-1})(f).
\end{split}
\end{eqnarray}

Therefore,
\begin{eqnarray*}
&& {\dlog}(\Phi_n^{St}(\gamma_1, \cdots \gamma_n)(f)) \\
&=& {\dlog}(\Phi_n^{St} (\gamma \rho^0, \cdots \gamma \rho^{n-1})(f)) \\
&=& \gamma \cdot {\dlog}(\mu_n^{St}(\gamma^{-1} \cdot f))\quad (\text{by Lemma \ref{lemt}}) \\
&=& \gamma \cdot ((-1)^n \cdot \mu_n^{NSh} (\widehat{\gamma^{-1} \cdot f}) \cdot \omega_T ) \quad (\text{by Lemma \ref{lemo}})\\
&=& (-1)^n \cdot \gamma \cdot (\mu_n^{NSh} (|\det \gamma|^{-1} \cdot \gamma^T \cdot \hat f ) \cdot \omega_T )\quad(\text{by Proposition \ref{kpf}}) \\
&=& (-1)^n |\det \gamma|^{-1} \cdot \gamma \cdot (\mu_n^{NSh}(\gamma^T \cdot \hat f) \cdot \omega_T ) \\
&=& (-1)^n |\det \gamma|^{-1} \cdot (\det \gamma) \cdot (\gamma \cdot \mu_n^{NSh} (\gamma^T \cdot \hat f ) ) \cdot \omega_T \\
&=& (-1)^n (\on{sign} \gamma) (\gamma \cdot \mu_n^{NSh} (\gamma^T \cdot \hat f ) ) \cdot \omega_T \\
&=& (-1)^n \mu_n^{NSh}|_{\gamma^{T}} (\hat f) \cdot \omega_T   \quad (\text{by definition of group action on measures})\\
&=& (-1)^n \Phi_n^{NSh}(\gamma \rho^0, \cdots,\gamma \rho^{n-1})(\hat f) \cdot \omega_T   \quad (\text{by Proposition \ref{Sproperty}})\\
&=& (-1)^n \Phi_n^{NSh}(\gamma_1, \cdots, \gamma_n)(\hat f) \cdot \omega_T \quad ( \text{by \ref{idy}}).
\end{eqnarray*}
\end{proof}

% \section{$p$-adic $L$-functions for totally real fields}
% Jonghyung, you are supposed to write this section.
% \subsection{The Shintani cocycle and Stele's work}
% Include the brief review of the Stele's work which constructs the $p$-adic $L$-function for totally real fields
% out of the $p$-adic Shintani cocycle.
% \subsection{The universal Eisenstein symbol and $p$-adic $L$-functions}

\section{Appendix}\label{sec4}

%\subsection{Appendix A. Basic notions on distributions}
%
%Include the basic notions regarding distributions based on \cite{C}.

\subsection{Appendix I: A Summary of Group Actions}
\label{sec4.1}

There are several $\GL(V)$-actions in this article. For reader's convenience, we summarize our convention for group actions. We use a left action for test functions and a right action for distributions with values in any right $\GL(V)$-module $M$:
\begin{eqnarray*}
\GL(V) \times \cS(V) &\to& \cS(V) \\
(\g, f) &\mapsto & \g \cdot f, \quad (\g \cdot f)(x):= f(x\cdot \g),
\end{eqnarray*}
\begin{eqnarray*}
 \Hom(\cS(V),M)\times \GL(V) &\to& \Hom(\cS(V),M) \\
(\mu, \g) &\mapsto &  \mu|_\g, \quad (\mu|_\g)(f):= \mu(\g \cdot f)|_\g.
\end{eqnarray*}
If $M$ is a left $\GL(V)$-module, then we instead use the action convention:
\begin{eqnarray}\label{lac}
(\mu|_\g)(f):=\g^T \cdot \left(\mu(\g \cdot f)\right).
\end{eqnarray}
We use a right action for trigonometric functions and relevant $K$-groups:
\begin{eqnarray*}
 \cR(V)\times \GL(V) &\to& \cR(V) \\
(F,\g) &\mapsto & F|_\g, \quad (F|_\g)(x):= F(x\cdot \g^{-1}),
\end{eqnarray*}\begin{eqnarray*}
 K_n^M(V) \times \GL(V) &\to& K_n^M(V) \\
(\{F_1, \cdots, F_n\},\g) &\mapsto & \{F_1, \cdots, F_n\}|_\g, \quad (\{F_1, \cdots, F_n\}|_\g:= \{F_1|_\g, \cdots, F_n|_\g\}.
\end{eqnarray*}

%We use a left action on a dual vector space $V^*$ and the group $\tilde C_k(V)$ of simplicial cones of dimension $k$:
%\begin{eqnarray*}
%\GL(V) \times V^* &\to& V^*:=\Hom(V, \bQ) \\
%(\g, \l) &\mapsto & \g \cdot \l, \quad (\g \cdot \l)(x):= \l(x \cdot \g).
%\end{eqnarray*}
%\begin{eqnarray*}
%\GL(V) \times \tilde C_k(V) &\to& \tilde C_k(V) \\
%(\g, \la\l_0\ra \wedge \cdots \wedge \la\l_{k}\ra) &\mapsto & \g \cdot ( \la\l_0\ra \wedge \cdots \wedge \la\l_{k}\ra), \quad \g \cdot ( \la\l_0\ra \wedge \cdots \wedge \la\l_{k}\ra):= \la \g \cdot \l_0\ra \wedge \cdots \wedge \la \g \cdot \l_{k}\ra.
%\end{eqnarray*}
%
%Finally, we use a right action on $\Hom(\tilde C_k(V),M)$ for any right $\GL(V)$-module $M$:
%\begin{eqnarray*}
% \Hom(\tilde C_k(V),M)\times \GL(V) &\to& \Hom(\tilde C_k(V),M) \\
%(\xi, \g) &\mapsto &  \xi|_\g, \quad (\xi|_\g)(\l):= \mu(\g \cdot \l)|_\g.
%\end{eqnarray*}

We use a left action on a dual vector space $V^*$ and the free abelian group $C_k(V)$ generated by $k$-simplexes:
\begin{eqnarray*}
\GL(V) \times V^* &\to& V^*:=\Hom(V, \bQ) \\
(\g, \l) &\mapsto & \g \cdot \l, \quad (\g \cdot \l)(x):= \l(x \cdot \g).
\end{eqnarray*}
\begin{eqnarray*}
\GL(V) \times C_k(V) &\to& C_k(V) \\
(\g, [\l_0, \cdots, \l_{k}]) &\mapsto & \g \cdot [\l_0, \cdots, \l_{k}], \quad \g \cdot [\l_0, \cdots, \l_{k}]:= [\g\cdot \l_0, \cdots, \g \cdot \l_{k}].
\end{eqnarray*}
Finally, we use a right action on $\Hom(C_k(V),M)$ for any right $\GL(V)$-module $M$:
\begin{eqnarray*}
 \Hom(C_k(V),M)\times \GL(V) &\to& \Hom(C_k(V),M) \\
(\xi, \g) &\mapsto &  \xi|_\g, \quad (\xi|_\g)([\l]):= \mu(\g \cdot [\l])|_\g.
\end{eqnarray*}

Let $\GL_n(\Q)$ act on $\bC((T_1, \cdots T_n ))$ and $\bC((T_1, \cdots T_n ))\omega_T$ on the left (where $\omega_T = dT_1 \wedge \cdots \wedge dT_n$) by 
\begin{align*}
& \gamma \cdot f =f((T_1, \cdots, T_n)\cdot \gamma)
= f(T \cdot \gamma^{(1)}, \cdots T \cdot \gamma^{(n)}) \\
& \gamma \cdot (f \omega_T) =f((T_1, \cdots, T_n)\cdot \gamma) (\det \gamma) \omega_T 
= f(T \cdot \gamma^{(1)}, \cdots T \cdot \gamma^{(n)}) (\det \gamma) \omega_T
\end{align*}
where $T=(T_1, \cdots, T_n)$ and $\gamma^{(j)}$ is the $j$th column of $\gamma$. 
%The action is well-defined because
%\begin{align*}
%    \gamma_1 \cdot \gamma_2 \cdot (f \omega_T) =& \gamma_1 \cdot f([z_1, \cdots z_n] \gamma_2)(\det \gamma_2) \omega_T \\
%    =& f([z_1, \cdots z_n] \gamma_1 \gamma_2)(\det \gamma_1 )( \det \gamma_2) \omega_T \\
%    =& (\gamma_1 \gamma_2) \cdot (f\omega_T)
%\end{align*}
%The determinant multiplication makes sense because heuristically
%\begin{align*}
%    & \gamma \cdot \d z_j = \gamma_{1j} \d z_1 + \cdots + \gamma_{nj} \d z_n \\
%    \implies & \gamma \cdot \omega_z = (\gamma_{11} \d z_1 + \cdots + \gamma_{n1} \d z_n) \wedge \cdots \wedge (\gamma_{1n} \d z_1 + \cdots + \gamma_{nn} \d z_n) = (\det \gamma) \omega_z \\
%    \implies & \gamma \cdot \omega_T = (2\pi i)^n (\det \gamma) \omega_z = (\det \gamma) \omega_T
%\end{align*}

%which defines an action because
%\begin{align*}
%    (\mu|_{\gamma_1}|_{\gamma_2}) (f) =& (\sign \gamma_2) \gamma_2^{T} \cdot (\mu|_{\gamma_1})(\gamma_2 \cdot f) \\
%    =& (\sign \gamma_1 \gamma_2) \gamma_2^T \gamma_1^T \mu(\gamma_1 \cdot \gamma_2 \cdot f) \\
%    =& \mu|_{\gamma_1 \gamma_2} (f)
%\end{align*}
%Let's show that this action is $\GL_n$-equivariant.
The group $\GL(V)$ acts on $\cL_\bQ$ (the abelian group generated by characteristic functions of open rational cones in $\bF^n$, modulo constant functions) on the left:
\begin{eqnarray}\label{sgc}
\begin{split}
 \GL(V) \times \cL_\bQ &\to \cL_\bQ & \\
(\g,c) &\mapsto   (\g \cdot c)(x):= \sign{\g} \cdot c (x \cdot \g),&
\end{split}
\end{eqnarray}
where we use the embedding $\bQ^n \subset \bR^n \subset \bF^n$ given in \cite{H}.

\begin{prop} \label{Sproperty}
The Shintani cocycle satisfies the following $\GL_n(\bQ)$-equivariance:
    \begin{align*}
        \Phi_n^{Sh}(\gamma \gamma_1 ,\cdots \gamma \gamma_n) = (\sign \gamma) \Phi_n^{Sh}(\gamma_1, \cdots \gamma_n) |_{\gamma^{T}}, \quad   
        \Phi_n^{NSh}(\gamma \gamma_1 ,\cdots \gamma \gamma_n) = \Phi_n^{NSh}(\gamma_1, \cdots \gamma_n) |_{\gamma^{T}}.
    \end{align*}
\end{prop}
\begin{proof}
 Since we have
    \begin{align*}
        \Phi_n^{Sh}(\gamma\gamma_1, \cdots \gamma\gamma_n)(f) =& \langle \sigma_n^{Sh}(\gamma\gamma_1, \cdots \gamma\gamma_n), f \rangle_{SH} = \langle \gamma^{-T} \cdot \sigma_n^{Sh}(\gamma_1, \cdots \gamma_n), f \rangle_{SH}, \\
    &   (\text{by Proposition 3 (ii) and the definition of $(\a*c)(v)$ in subsection 3.1, \cite{H}})\\
           \Phi_n^{Sh}(\gamma_1, \cdots \gamma_n) |_{\gamma^{T}}(f)=& \gamma^{} \cdot \Phi_n^{Sh}(\gamma_1, \cdots \gamma_n)(\gamma^{T} \cdot f) = \gamma^{} \cdot \langle \sigma_n^{Sh}(\gamma_1, \cdots \gamma_n), \gamma^{T} \cdot f \rangle_{SH},
%       \Phi_n^{Sh}(\gamma_1, \cdots \gamma_n) |_{\gamma^{-1}}(f)=& (\sign \gamma) \gamma^{-T} \cdot \Phi_n^{Sh}(\gamma_1, \cdots \gamma_n)(\gamma^{-1} \cdot f) =(\sign \gamma)  \gamma^{-T} \cdot \langle \sigma_n^{Sh}(\gamma_1, \cdots \gamma_n), \gamma^{-1} \cdot f \rangle_{SH},
    \end{align*}
    it suffices to show that 
    \begin{align}\label{sgb}
    \langle \gamma \cdot c, f \rangle_{SH} = (\sign \gamma) \gamma^{-T} \cdot \langle c, \gamma^{-1} f \rangle_{SH}
    \end{align}
     for each cone function $c$. 
    Let $f = \chi_L$ with $L= \ZZ v_1 + \cdots + \ZZ v_r$ and $\gamma \cdot c = \chi_S$ with $S = \RR^+ v_1 + \cdots + \RR^+ v_r$ with $v_j \in L$. 
    This implies $\gamma^{-1} f = (v \mapsto f(v \cdot \gamma^{-1})) = \chi_{L \cdot \gamma}$ and $c = \gamma^{-1} (\gamma \cdot c) = (\sign \gamma) \cdot (v \mapsto \chi_S(v \cdot \gamma^{-1})) = (\sign \gamma) \cdot \chi_{S \cdot \gamma}$. Let $P_{c} = (0,1] v_1 + \cdots + (0,1] v_r$. Then (with notation $X=(T_1, \cdots, T_n)$ to avoid confusion with the transpose $\g^T$ of the matrix $\g$)
    \begin{align*}
        \langle \gamma \cdot c, f \rangle_{SH} = \frac1{1-e^{{X} \cdot v_1}} \cdots \frac1{1-e^{{X} \cdot v_r}} \sum_{w \in P_{c}} f(w) e^{{X} \cdot w}
    \end{align*}
    and
    \begin{align*}
        (\sign \gamma) \gamma^{-T} \cdot \langle c, \gamma^{-1} f \rangle_{SH} =& (\sign \gamma) \gamma^{-T} \cdot \langle (\sign \gamma) \cdot \chi_{S \cdot \gamma}, \chi_{L\cdot \gamma} \rangle_{SH} \\
	=& (\sign \gamma)^2 \cdot \gamma^{-T} \cdot \left( \frac1{1-e^{{X} \cdot (v_1 \cdot \gamma)}} \cdots \frac1{1-e^{{X} \cdot (v_r \cdot \gamma)}} \sum_{w \in P_{c} \cdot \gamma} (\gamma^{-1} f)(w) e^{{X} \cdot w} \right) \\
        =& \gamma^{-T} \cdot \left( \frac1{1-e^{{X} \cdot (v_1 \cdot \gamma)}} \cdots \frac1{1-e^{{X} \cdot (v_r \cdot \gamma)}} \sum_{w \in P_{c}} (\gamma^{-1} f )(w \cdot \gamma) e^{{X} \cdot (w \cdot \gamma)} \right) \\
        =& \gamma^{-T} \cdot \left( \frac1{1-e^{{X} \cdot (v_1\cdot \gamma)}} \cdots \frac1{1-e^{{X} \cdot (v_r \cdot \gamma)}} \sum_{w \in P_{c}} f(w) e^{{X} \cdot (w \cdot \gamma)} \right) \\
        =& \frac1{1-e^{({X} \cdot \gamma^{-T}) \cdot (v_1\cdot \gamma)}} \cdots \frac1{1-e^{({X} \cdot \gamma^{-T}) \cdot (v_r \cdot \gamma)}} \sum_{w \in P_{c}} f(w) e^{({X} \cdot \gamma^{-T}) \cdot (w \cdot \gamma)} \\
        =& \frac1{1-e^{{X} \cdot v_1}} \cdots \frac1{1-e^{{X} \cdot v_r}} \sum_{w \in P_{c}} f(w) e^{{X} \cdot w} \\
        =& \langle \gamma \cdot c, f \rangle_{SH}.
    \end{align*}

%The $\GL_n(\bQ)$-equivariance for $\Phi_n^{NSh}$ follows from the same equality $\langle \gamma \cdot c, f \rangle_{SH} = (\sign \gamma) \gamma^{-T} \cdot \langle c, \gamma^{-1} f \rangle_{SH}$ combined with the observation
%$$
%\left(\g^{-T}\cdot (\chi_{\RR^+ e_1(\rho^0)^T+ \cdots + \RR^+e_1(\rho^{(n-1)})^T}) \right)(x) = (\chi_{\RR^+ e_1(\g\rho^0)^T+ \cdots + \RR^+e_1(\g\rho^{(n-1)})^T})(x),   \g \in \GL_n(\bQ), \ x \in \bR^n\subset \bF^n.
%$$
   The $\GL_n(\bQ)$-equivariance for $\Phi_n^{NSh}$ follows from the following computation:
  \begin{align*}
  \Phi_n^{NSh}(\g\g_1, \cdots, \g\g_n) (f) 
  =& \langle \chi_{\bR^+e_1 (\g \g_1)^T + \cdots +\bR^+e_1 (\g \g_n)^T }, f \rangle_{SH}      \\  
  =&\sign(\gamma) \langle \g^{-T} \cdot \chi_{\bR^+e_1 (\g_1)^T + \cdots + \bR^+e_1 (\g_n)^T }, f \rangle_{SH}   \quad(\text{by \ref{sgc}})     \\
  =&\sign(\gamma)^2 \g \cdot \langle \chi_{\bR^+e_1 (\g_1)^T + \cdots + \bR^+e_1 (\g_n)^T } , \g^T\cdot f\rangle_{SH} \quad (\text{by \ref{sgb}})     \\
  =&\g \cdot \left( \Phi_n^{NSh}(\g_1, \cdots, \g_n) \right)(\g^T\cdot f) \quad (\text{by definition})\\
  =&\left( \Phi_n^{NSh}(\g_1, \cdots, \g_n)|_{\g^T} \right)(f) \quad(\text{by \ref{lac}}).
  \end{align*}  
\end{proof}

\subsection{Appendix II: Fourier Theory on $\cS(V)$ }
\label{sec4.2}

Here we summarize some results on the Fourier theory on $\cS(V)$ with the Lattice Topology. We refer to \cite{C} for more details on naive distributions and the Fourier theory on $\cS(V)$.

As $\iota_{V,W}: \cS(V) \otimes \cS(W) \cong \cS (V \times W)$, we abuse notation and sometimes denote $\iota_{V,W}(f \otimes g)$ by $f \otimes g$.
Let $\textbf{h}_{L_0}$ be the unique translation-invariant distribution satisfying $\textbf h_{L_0}(L_0) = 1$. This gives, for example, $\textbf h_\ZZ(\chi_{a+d\ZZ}) = \frac1d$. 
For a $n$-dimensional $\QQ$-vector space $V$, a lattice $L \subseteq V$, a symmetric non-degenerate bilinear pairing $\langle \cdot,\cdot \rangle$ and a $\QQ^{\on{ab}}$-algebra $R$, define the Fourier transform by
\begin{align*}
    & \cF: \cS(V,R) \rightarrow \cS(V,R) \\
    & \cF(y) = \hat f(y) := \int_V f(x) e^{-2\pi i\langle x,y \rangle} \mathrm{d} \textbf h_{L}(x).
\end{align*}
In our case, we assume that $V = \QQ^n, L = \ZZ^n, \langle (x_1, \cdots x_n),(y_1, \cdots y_n) \rangle = \sum_i x_i y_i$. 
We only state the necessary propositions without proofs except for Proposition \ref{kpf}.
%We omit the proofs of the standard results except for Proposition \ref{kpf}.

\begin{prop}
    For $f \in \cS(\QQ^m), g \in \cS(\QQ^n)$,
    $$\cF(f \otimes g) = \cF(f) \otimes \cF(g).$$
\end{prop}
%\begin{proof}
%Left to the readers.
%\end{proof}
%\begin{proof}
%    Let $V = \QQ^m, W = \QQ^n$. For $z \in V \times W$, denote by $z_1, z_2$ its $V$- and $W$-components. For $f = \chi_{a_1 + d_1 \ZZ^m}, g = \chi_{a_2 + d_2 \ZZ^n}$, we have
%    \begin{align*}
%        & \int_{V \times W} f(x_1) g(x_2) \mathrm{d} \textbf{h}_{\ZZ^{m+n}} (x) = \int_{V \times W} \chi_{(a_1 + d_1 \ZZ^m) \times (a_2 + d_2 \ZZ^n)} \mathrm{d} \textbf{h}_{\ZZ^{m+n}}(x)
%        = \frac1{d_1^m} \frac1{d_2^n} \\
%        & \int_V f(x_1) \mathrm{d} \textbf{h}_{\ZZ^{m}} (x) \int_V g(x_2) \mathrm{d} \textbf{h}_{\ZZ^{n}} (x) = \frac1{d_1^m} \frac1{d_2^n}
%    \end{align*}
%    By extending linearly, we see that this claim holds without assuming a specific form on $f,g$. Now
%    \begin{align*}
%        \cF(f \otimes g)(y) =& \int_{V \times W} (f\otimes g)(x) e^{-2\pi i \langle x, y \rangle} \mathrm{d} \textbf{h}_{\ZZ^{m+n}}(x) \\
%        =& \int_{V \times W} f(x_1) e^{-2\pi i \langle x_1, y_1 \rangle} g(x_2) e^{-2\pi i \langle x_2, y_2 \rangle} \mathrm{d} \textbf{h}_{\ZZ^{m+n}}(x) \\
%        =& \int_{V} f(x_1) e^{-2\pi i \langle x_1, y_1 \rangle} \mathrm{d} \textbf{h}_{\ZZ^{m}}(x) \int_{W} g(x_2) e^{-2\pi i \langle x_2, y_2 \rangle} \mathrm{d} \textbf{h}_{\ZZ^{n}}(x) \\
%        =& \cF(f)(y_1) \cdot \cF(g)(y_2) \\
%        =& (\cF(f) \otimes \cF(g))(y)
%    \end{align*}
%\end{proof}

\begin{prop}
    For $L = (a_1 + d_1 \ZZ) \times \cdots \times (a_n + d_n \ZZ)$ and $L' = (\frac1{d_1} \ZZ) \times \cdots \times (\frac1{d_n} \ZZ)$,
    \begin{align*}
        \widehat{\chi_L} (y) = \frac1{d_1 \cdots d_n} e^{-2\pi i \langle a, y \rangle} \chi_{L'}(y).
    \end{align*}
%    In particular,
%    \begin{align*}
%        \widehat{\chi_{a+d\ZZ}} (y) = \frac1d e^{-2\pi iay} \chi_{\frac1d \ZZ}(y)
%    \end{align*}
\end{prop}
%\begin{proof}
%We first prove the 1-dimensional case. Let $d = \frac{d_1}{d_2}, y = \frac{y_1}{y_2}$. Then
%    \begin{align*}
%        \hat{\chi}_{a + d \ZZ}(y)
%        =& h_{\ZZ}(\chi_{a+d\ZZ}(x)e^{-2\pi ixy}) \\
%        =& h_{\ZZ}(\sum_{j=0}^{d_2 y_2-1} \chi_{a + jd + d_1y_2 \ZZ}(x) e^{-2\pi ixy}) \\
%        =& h_{\ZZ}(\sum_{j=0}^{d_2 y_2-1} \chi_{a + jd + d_1y_2 \ZZ}(x) e^{-2\pi i(a + jd) y}) \\
%        =& \sum_{j=0}^{d_2 y_2-1} h_{\ZZ}(\chi_{a + jd + d_1y_2 \ZZ}(x)) e^{-2\pi i(a + jd) y}\\
%        =& \sum_{j=0}^{d_2 y_2-1} \frac1{d_1 y_2} e^{-2\pi i(a + jd) y} \\
%        =& \frac1{d_1y_2} e^{-iay} \sum_{j=0}^{d_2 y_2-1} (e^{-2\pi id y})^j \\
%        =& \begin{cases}
%        \frac1d e^{-2\pi ia y} & \text{ if $y \in \frac1d \ZZ$} \\
%        0 & \text{ otherwise}
%        \end{cases} \\
%        =& \frac1d e^{-2\pi ia y} \chi_{\frac1d \ZZ}(y)
%    \end{align*}
%    where we used the fact that $\sum_j (e^{-2\pi id y})^j=0$ unless $e^{-2\pi id y}=1 \iff dy \in \ZZ \iff y \in \frac1d \ZZ$.
%    
%    More generally, we now have
%    \begin{align*}
%        \widehat{\chi_L}(y) =& \cF( \chi_{a_1 + d_1 \ZZ} \otimes \cdots \otimes \chi_{a_n + d_n \ZZ} )(y) \\
%        =& \cF(\chi_{a_1 + d_1 \ZZ})(y_1) \otimes \cdots \cF(\chi_{a_n + d_n \ZZ})(y_n) \\
%        =& \frac1{d_1} e^{-2\pi i a_1 y_1} \chi_{\frac1{d_1} \ZZ}(y_1) \otimes \cdots \otimes \frac1{d_n} e^{-2\pi i a_n y_n} \chi_{\frac1{d_n} \ZZ}(y_n) \\
%        =& \frac1{d_1 \cdots d_n} e^{-2\pi i \langle a, y \rangle} \chi_{L'}(y)
%    \end{align*}
%\end{proof}

\begin{prop}\label{lp}
    For $V = \QQ^n$, a $\QQ^{ab}$-algebra $R$, $f \in \cS(V)$, and $\gamma \in \GL(V)$, 
    \begin{align*}
        \int_V f( x \cdot \gamma) \mathrm{d} \textbf h_{\ZZ^n} (x) = |\det \gamma| \cdot \int_V f(x) \mathrm{d}\textbf h_{\ZZ^n} (x).
    \end{align*}
\end{prop}
%\begin{proof}
%Left to the readers.
%\end{proof}
%\begin{proof}
%    \textcolor{blue}{(Skip for now)}
%\end{proof}

\begin{prop} \label{kpf}
    For $V = \QQ^n$, $f \in \cS(V)$ and $\gamma \in \GL(V)$,
    $$\widehat{\gamma \cdot f} = |\det \gamma| \cdot ((\gamma^{-1})^T \cdot \hat f).$$
\end{prop}
\begin{proof}
    \begin{align*}
        \widehat{\gamma \cdot f}(y) =& \int_V (\gamma \cdot f)(x) e^{-2\pi i \langle x,y \rangle} \mathrm{d} \textbf h_{\ZZ^n}(x) \\
        =& \int_V f(x \cdot \gamma) e^{-2\pi i \langle x \cdot \gamma \cdot \gamma^{-1},y \rangle} \mathrm{d} \textbf h_{\ZZ^n}(x)  \\
        =& |\det \gamma| \int_V f(x) e^{-2\pi i \langle x \cdot \gamma^{-1},y \rangle} \mathrm{d} \textbf h_{\ZZ^n}(x) 
        \quad \text{ (Proposition \ref{lp}) }\\
        =& |\det \gamma| \int_V f(x) e^{-2\pi i \langle x ,y \cdot (\gamma^{-1})^T \rangle} \mathrm{d} \textbf h_{\ZZ^n}(x) \\
        =& |\det \gamma| \hat f (y \cdot (\gamma^{-1})^T ) \\
        =& |\det \gamma| ((\gamma^{-1})^T \cdot \hat f)(y).
    \end{align*}
\end{proof}


\begin{thebibliography}{PTW02} % '2nd argument contains the widest acronym'



\bibitem{Ce} 
Busuioc, C.: The Steinberg symbol and special values of $L$-functions. Trans. Am. Math. Soc. 360(11), 5999--6015 (2008)



%\bibitem{AdSp06}  A. Adolphson,  S. Sperber, \emph{On the Jacobian ring of a complete intersection}, J. Algebra 304 (2006), no. 2, 1193-1227. 
\bibitem{C}  Campbell, D.:Eisenstein distribution and $p$-adic $L$-functions. Ph.D.Thesis,Boston University(1997)


\bibitem{CDG1} Charollois, P., Dasgupta, S.: Integral Eisenstein cocycles on $\GL_n$, I: Sczech's cocycle and $p$-adic $L$-functions of totally real fields. Cambridge J. Math. 2, 4--90 (2014)

\bibitem{CDG2} Charollois, P., Dasgupta, S., Greenberg, M.: Integral Eisenstein cocycles on GLn, II: Shintani's method. Comment. Math. Helv. 90, 43--477 (2015)



\bibitem{Das} Dasgupta, S.:
Shintani zeta functions and Gross-Stark units for totally real fields, 
Duke Mathematical Journal, 143 (2008), no. 2, 225--279

\bibitem{Spi} 
 Spiess, M.: Shintani cocycles and the order of vanishing of $p$-adic Hecke $L$-series at $s = 0$. Math.
Ann. 359(1-2), 23--265 (2014)


%Dasgupta, Spiess


\bibitem{H} Hill, R.: Shintani cocycles on $\GL_n$. Bull. London Math. Soc. 39, 99--1004 (2007)



\bibitem{Solo}  Hu, S., Solomon, D.: Properties of higher-dimensional Shintani generating functions and cocycles on $\operatorname{PGL}_3(\bQ)$. Proc. London Math. Soc. 82, 6--88 (2001)



\bibitem{K} Kerr, M.: A regulator formula for Milnor $K$-groups. $K$-Theory 29: 17--210, 2003

\bibitem{P}
Park, J.: Milnor $K_2$ and $p$-adic $L$-functions for real quadratic fields. Ann. Math. Qu\'ebec (2017) 41:3-25 

\bibitem{Sh} Sharifi, R.: Modular symbols and Milnor $K_2$, http://math.ucla.edu/~sharifi/milnork2.pdf

\bibitem{AS}
 Steele, G. A.:The $p$-adic Shintani cocycle. Math. Res. Lett. 21(2), 40--422 (2014)
%Ander Stele, The $p$-adic Shintani cocycle

\bibitem{S}
Stevens, G.:$K$-theory and Eisenstein series. a preprint available in http://math.bu.edu/people/ghs/research.html

%
%\bibitem{St}
%Stevens, G.:The Eisenstein measure and real quadratic fields. In:de Koninck,J.M.,Levesque,C.(eds.) The proceedings of the international number theory conference (Universit\'e Laval, 1987), de Gruyter (1989)



%\bibitem{VM} 
%Philippe Elbaz-Vincent, Stefan M\"{u}ller-Stach, Milnor $K$-theory of rings, higher Chow groups 
%and applications.



\end{thebibliography}
\end{document}